\documentclass[reqno,english]{amsart} 
\usepackage{amsmath}
\usepackage[colorlinks=true]{hyperref}
\hypersetup{urlcolor=blue, citecolor=red}
\usepackage{tikz-cd}
\usepackage{comment}
\usepackage{hyperref}

\usepackage[square,numbers]{natbib}

\newcommand{\RR}{\mathbb{R}}
\newcommand{\TT}{\mathbb{T}}

\newcommand{\SSS}{\mathbb{S}}
\newcommand{\ggg}{\mathsf{g}}
\newcommand{\ca}[1]{\mathcal{#1}}
\newcommand{\mrm}[1]{\mathrm{#1}}

\newcommand{\dd}{\, {\rm d}}
\newcommand{\ve}{\varepsilon}
\newcommand{\de}{\delta}

\newcommand{\iy}{\infty}
\newcommand{\p}{\partial}

\newcommand{\na}{\nabla}

\newcommand{\norm}[1]{\left\|#1\right\|}

\newcommand{\set}[1]{\left \{#1\right \} }

\DeclareMathOperator*\esssup{ess\,sup}

\newtheorem{theorem}{Theorem}[section]

\newtheorem{lemma}[theorem]{Lemma}
\newtheorem{proposition}[theorem]{Proposition}

\theoremstyle{definition}

\newtheorem*{remark*}{Remark}

\newcounter{cases}
\newcounter{subcases}[cases]

\title[{Global Well-Posedness for 1D Boltzmann}]{{Global Well-Posedness and Large Data Estimates for the 1D Boltzmann Equation}}

\author{Dominic Wynter}

\begin{document}

	\begin{abstract}
		We prove quantitative growth estimates for large data solutions to the 1D Boltzmann equation, for a collision kernel with angular cutoff and relative velocity cutoff. We present proofs for the global well-posedness results presented in the note of Biryuk, Craig, and Panferov, in which global solutions for this equation are shown to exist for large data, with density bounded for all time. We show that these solutions propagate moments in $v$, and derivatives in $x$ and $v$. Our main contribution is to develop new, sharp integral inequality estimates, which allow us to prove exponential growth bounds in $L^\infty_x L^1_v$ for large data, and to prove dissipation in $ L^\infty_x L^1_v$ for finite energy data on the line. 
	\end{abstract}
	
	\maketitle

	\tableofcontents

	\section{Introduction}

	In this paper, we study the initial value problem for the Boltzmann equation in one spatial dimension,
	\begin{equation}\label{1D_Boltzmann}
		\begin{cases}
			\dfrac{\p f}{\p t} + v_1\dfrac{\p f}{\p x} = Q(f,f)	&
			\qquad \hbox{ for }
			(t,x,v)\in [0,\infty)\times\Omega_x^1\times\mathbb{R}_v^3	\\
			f(0,x,v) = f_\mrm{in}(x,v)\ge0 &
			\qquad \hbox{ for }
			(x,v) \in \Omega_x^1\times\mathbb{R}_v^3
		\end{cases}
	\end{equation}
	with spatial domain $\Omega_x^1=\RR_x^1$ or $\TT_x^1$, where $v = (v_1,v_2,v_3)\in\RR_v^3$ in Cartesian coordinates, and where $Q(f,f)$ is the Boltzmann collision operator, which we define in \eqref{nonsymmetric_Boltzmann_operator}. The existence theory of this equation is a well-studied problem, both in the spatially homogeneous setting and in three-dimensional inhomogeneous setting. In the homogeneous problem, an existence and uniqueness result was developed for data with polynomial decay by Carleman \cite{Carleman57}, and this result was generalized to polynomially weighted $L^1$-data by Arkeryd \cite{Arkeryd72}, and then by interpolation to $L^p$ data by Gustafsson \cite{Gustafsson88}. Constructive $W^{k,p}_\ell$ estimates were proved by Mouhot and Villani \cite{MouhotVillani04}, who also proved large time quantitative bounds and convergence to equilibrium. In the spatially inhomogeneous setting, large time perturbative results were obtained near Maxwellians by Ukai \cite{Ukai74}, and near vacuum by Illner and Shinbrot \cite{IllnerShinbrot84}. Existence of global in time large data solutions was proved by DiPerna and Lions, who developed a theory of renormalized solutions \cite{DiPernaLions89}, and a weak-strong uniqueness result in this setting was proved by Lions \cite{Lions94}. 
	
	For collision operators without angular cutoff, a theory of renormalized solutions with defect measure was obtained by Alexandre and Villani \cite{AlexandreVillani02}, and a well-posedness and stability theory was developed for the spatially homogeneous case by Desvillettes and Mouhot \cite{DesvillettesMouhot09}. The spatially inhomogeneous problem near global Maxwellians was addressed by Alexandre-Morimoto-Ukai-Xu-Yang \cite{AMUXY_1} \cite{AMUXY_2} \cite{AMUXY_3} and by Gressman and Strain \cite{GressmanStrain11}, who also proved rates of convergence towards equilibrium. More recently, a stability of vacuum result was obtained by Chaturvedi \cite{Chaturvedi21}. We will, however, not treat the non-cutoff case in this paper.
	
	The one dimensional problem \eqref{1D_Boltzmann} was first studied in stationary weak form in the papers \cite{ArkerydCercignaniIllner91} \cite{ArkerydNouri95} \cite{Cercignani98} \cite{ArkerydNouri00}. These results were extended to the non-stationary problem by Cercignani, who proved global weak solutions for $L^1(1+|v|^2)\cap L\log L$ initial data \cite{Cercignani05}.
	
	A global well-posedness theory for \eqref{1D_Boltzmann} was developed by Andrei Biryuk, Walter Craig, and Vladimir Panferov, who presented \cite{BCPslides} and published a summary of their result \cite{BiryukCraigPanferov}, though a full proof was never released. This theory is based off of a novel angular averaging estimate for the gain term of the collision kernel. For data close to global Maxwellians, the gain term can then be controlled by relative entropy, which then gives estimates that can be closed for all time. The theory for general large data relies on a Lyapunov functional particular to 1D kinetic models, developed by Tartar \cite{Tartar87} and Bony \cite{Bony87}. It was observed by Cercignani that this functional can be used to control the collision rate $\norm{Q^+(f,f)}_{L^1([0,T],L^1)}$ globally in time for suitably truncated collision kernels \cite{Cercignani92}, which was combined with an angular averaging estimate to get a large data theory for strong solutions in \cite{BiryukCraigPanferov}.
	
	In this paper, we present a proof of these global solution theories, and prove further quantitative estimates on the growth of solutions for large data, as well as a dissipation result for large data on $\RR_x^1\times\RR_v^3$. Our essential technical contribution is a sharp growth estimate for an integral inequality introduced in \cite{BiryukCraigPanferov}, which we use to prove our main result.
		
	We now set up our problem. The Boltzmann collision operator is defined by
	\begin{equation}
		\label{nonsymmetric_Boltzmann_operator}
		Q(g,f)(v) = \int_{\RR^3_{v_*}\!\times\mathbb{S}^2_\sigma}
		B\left( |v-v_*|,\frac{v-v_*}{|v-v_*|}\cdot\sigma\right)
		[g(v_*')f(v')-g(v_*)f(v)]\dd\sigma\dd v_*,
	\end{equation}
	for a collision kernel $B\ge0$, and where the velocities $v',v_*'$ are defined as
	\begin{equation}
		\label{post-collisional_velocities}
		v' = \frac{v+v_*}{2} + \sigma\frac{|v-v_*|}{2},
		\qquad
		v_*' = \frac{v+v_*}{2} -\sigma\frac{|v-v_*|}{2},
		\qquad \sigma\in\mathbb{S}^2.
	\end{equation}
	The velocities $v,v_*$ and $v',v_*'$ obey the momentum and energy conservation laws for elastic collisions
	\begin{align}
		\label{two_particle_conservation_laws}
		\begin{split}
			v + v_* &= v' + v_*' \\
			|v|^2 + |v_*|^2 &= |v'|^2 + |v_*'|^2.
		\end{split}
	\end{align}
	
	We refer to the textbook \cite{CercignaniIllnerPulvirenti94}, and the more recent review article \cite{Villani02}, for the standard theory and historical development of the Boltzmann equation. We consider collision kernels with a Grad angular cutoff, that is, $B\in L^1_{\mrm{loc}}$. More precisely, we have the following.
	
	\subsection{The Collision Kernel}
	We impose the following conditions on the collision kernel (adapted from \cite{BiryukCraigPanferov}).
	
	\begin{enumerate}
		\item[\textbf{(H1)}] We suppose that $B(r,\cos\theta)\le\phi(r)r$ for some non-increasing function $\phi\in L^1([0,\infty))$, and we require that
		\begin{align*}
			\norm{(1+r^{-1})B(r,\cos\theta)}_{L^\infty_{r,\theta}}<\infty.
		\end{align*}
		
		\item[\textbf{(H2)}] We have
		$B(r,\cos\theta)=0$ when $r\le R_0$, for some $R_0>0$. Additionally, we have 
		\begin{align*}
			\delta \sup_{\theta\in[0,\pi]} B(r,\cos\theta) 
			\le{2\pi}\int_0^\pi\! B(r,\cos\theta)\sin\theta\dd\theta
		\end{align*}
		for some $\delta>0$ independent of $r$.
		
	\end{enumerate}
	
	We note that hypothesis \textbf{(H1)} can be satisfied by $B(r,\cos\theta) = \Phi(r)b(\cos\theta)$ for bounded functions $\Phi,b$ such that
	\begin{align*}
		\Phi(r)\le r\phi(r) = C\frac{r}{1+r \log^{1+\ve}(1+r)},
	\end{align*}
	for some $C,\epsilon>0$ (which is the hypothesis in \cite{BiryukCraigPanferov}). If we additionally impose the low velocity cutoff condition $\Phi(r) = 0$ for $r\le R_0$, then the collision kernel also satisfies \textbf{(H2)}.
	
	We define the Sobolev space $W^{k,1}(\Omega^1\times\RR^3)$ for any integer $k\ge0$ using the norm
	\begin{equation*}
		\norm{f}_{W^{k,1}} =
		\sum_{\alpha + |\beta| \le k} 
		\int_{\Omega^1_x}\!\int_{\RR^3_v}\! 
		|\p_x^\alpha\p_v^\beta f(x,v)|
		\dd v \dd x
	\end{equation*}
	where $\alpha\in\mathbb{Z}_{\ge0},\beta\in\mathbb{Z}_{\ge0}^3$ are multi-indices. For any  $\ell\ge0$, we define the weight function
	\begin{align*}
		m_\ell(v) = (1+|v|^2)^{\ell/2}
	\end{align*}
	and define the weighted Sobolev space $W^{k,1}_\ell(\Omega^1\times\RR^3)$  for $k,\ell\ge0$ as
	\begin{equation*}
		\norm{f}_{W^{k,1}_\ell} = \norm{f m_\ell}_{W^{k,1}_\ell},
	\end{equation*}
	and we set $L^1_\ell(\Omega^1\times\RR^3) = W^{0,1}_\ell(\Omega^1\times\RR^3)$. To use standard notation, we will also write $L^1_2$ as $L^1(1+|v|^2)$.
	
	We say that $f\in L\log L$ when 
	\begin{align*}
		\int_{\Omega_x^1\times\RR_v^3} |f(x,v)|\log^+|f(x,v)|\dd v\dd x <\infty
	\end{align*}
	where we write $\log^+ = \max(0,\log^+)$ and $\log = \log^+-\log^-$. For any nonnegative function $f\in L\log L\cap L^1(1+|v|^2)$, we can define the entropy
	\begin{align*}
		H[f] = \int_{\Omega_x^1\times\RR_v^3} f(x,v)\log f(x,v)\dd v\dd x.
	\end{align*}
	For any nonnegative $f,g\in L^1(\Omega_x^1\times\RR_v^3)$ with the same mass, we can define the relative entropy
	\begin{align*}
		H(f|g) = \int_{\Omega_x^1\times\RR_v^3} f(x,v)\log \frac{f(x,v)}{g(x,v)}\dd v\dd x\ge0,
	\end{align*}
	where nonnegativity follows from convexity. It is known since Boltzmann \cite{Boltzmann72} that the quantity $H[f(t)]$ is non-increasing for solutions $f$ of the Boltzmann equation. Defining the entropy production functional
	\begin{align*}
		D_H(t) &=
		\frac14
		\int_{\Omega^1}\!\int_{\RR^3}\!\int_{\RR^3}\!\int_{\mathbb{S}^2}
		B\left(|v-v_*|,\sigma\cdot\frac{v-v_*}{|v-v_*|}\right)
		[f' f_*' - f f_*]\log\frac{f' f_*'}{f f_*}
		\dd\sigma\dd v_*\dd v\dd x.
		\\
		&\ge0,
	\end{align*}
	and writing $H(t)$ for $H[f(t)]$, we have the classical entropy inequality
	\begin{align}\label{entropy_inequality}
		H(t) + \int_0^t D_H(s)\dd s \le H[f_{\mrm{in}}] .
	\end{align}
	We also have the classical conservation laws
	\begin{align}\label{conservation_laws}
		\frac{\mrm{d}}{\mrm{d}t}
		\int_{\Omega_x^1}\!\int_{\RR_v^3}\!
			f(t,x,v)
			\begin{pmatrix}
				1 \\ v \\ |v|^2
			\end{pmatrix} \dd v\dd x = 0,
	\end{align}
	which follow from the cancellation identities
	\begin{align}\label{collision_kernel_cancellation}
		\int_{\RR_v^3}\! Q(f,f)(v)
		\begin{pmatrix}
			1 \\ v \\ |v|^2
		\end{pmatrix} \dd v = 0.
	\end{align}
	
	We define the norm
	\begin{align*}
		\norm{f}_X = \sup_{q\ge0}\norm{S_q f}_{L^\infty_x(\Omega_x^1, L^1(\RR_v^3))},
	\end{align*}
	where $S_q f(x,v) = f(x-qv_1,v)$ is the transport semigroup. Using the classical notation $f^\#(t,x,v) =  f(t,x+tv_1,v)$, we rewrite the Boltzmann equation in mild form, that is
	\begin{align}\label{mild_Boltzmann}
		\p_t f^\# = Q(f,f)^\#,
	\end{align}
	which is well defined whenever $Q(f,f)^\#(t,x,v)\in L^1([0,T])$ for almost every $x,v$. Using that $f^\#(t,x,v) = (S_{-t}f)(t,x,v)$, we integrate \eqref{mild_Boltzmann} in time and expand to get
	\begin{align*}
		S_{-t}f(t) = f_{\mrm{in}} + \int_0^t S_{-s}Q(f,f)(s)\dd s
	\end{align*}
	and by applying $S_t$ to both sides, we see that the Duhamel formulation
	\begin{align}\label{Duhamel_Boltzmann}
		f(t) = S_tf_{\mrm{in}}+ \int_0^t S_{t-s} Q(f,f)(s)\dd s
	\end{align}
	is equivalent to  \eqref{mild_Boltzmann} for any mild solution, and we will study the Boltzmann equation in this form.
	
	We now present global well-posedness theories for this problem. In Theorem \ref{Thm1}, we present a global well-posedness theory for initial data $f_{\mrm{in}}$ with small relatively entropy $H(f_{\mrm{in}}|M)$ for some Maxwellian $M$, and we prove quantiative growth rates for the solution. This theorem is adapted directly from \cite{BiryukCraigPanferov}, and we make no claim of originality for this result. We hope, however, that a full proof of this theorem will be useful for the literature. We compare this result to the nonlinear stability result for the hard-sphere Boltzmann equation proved in \cite{GualdaniMischlerMouhot17} on the space $\TT_x^3\times\RR_v^3$. While we are able to improve on the $\norm{f_{\mrm{in}}-M}_{L^1_v L^\infty_x(1+|v|^{2+\delta})}$ smallness condition in our one-dimensional context, we lose the nonlinear stability result $\norm{f(t)-M}_{L^1_v L^\infty_x(1+|v|^{2+\delta})} = O(e^{-\lambda t})$ of \cite{GualdaniMischlerMouhot17}, and we can only prove subexponential growth of norms. This loss of nonlinear stability comes from our inability to fully exploit the loss term $Q^{-}(f,f)$, which we use only to prove the conservation laws \eqref{collision_kernel_cancellation} and entropy dissipation \eqref{entropy_inequality}. While one might hope to exploit the loss term $Q^-(f,f)$ to prove a nonlinear trapping result, it is not clear if this will hold in the $X$ norm, and so we treat the loss term only as a perturbation. 
	
	In Theorem \ref{Thm2}, we present a global well-posedness theory for large data $f_{\mrm{in}}\in X\cap L^1(1+|v|^2)$, improving on \cite{BiryukCraigPanferov} to prove quantitative growth rates in the $X$ norm. This theorem relies essentially on a Lyapunov functional introduced by Tartar \cite{Tartar87} and Bony \cite{Bony87} for one-dimensional discrete velocity collisional models, which was then generalized to continuous velocity by Cercignani \cite{Cercignani92}. Using this functional, which we call the Bony functional, we are then able to close an integral inequality on the $X$ norm by following estimates in \cite{BiryukCraigPanferov}. We then show sharp growth estimates for this integral inequality, which allow us to prove exponential growth of the norm $\norm{f(t)}_X$ on the torus $\TT^1_x$, as well as dissipation $\norm{f(t)}_X\to 0$ on the whole space $\RR^1_x$. We note that the difference in velocity weights in the $X$ norm compared to the Bony functional mean that the collision kernel must vanish for small relative velocities for the result to hold. We now state the main theorems.
	
	\begin{theorem}[Global well-posedness near Maxwellians -- adapted from \cite{BiryukCraigPanferov}]\label{Thm1}
		Let $B$ be a collision kernel satisfying \textbf{(H1)}, and take initial data $f_{\mrm{in}}\in X\cap L^2_1\cap L\log L(\TT^1_x\times\RR_v^3)$. Let $m = \int f_{\mrm{in}}\dd v\dd x$ be its mass. Suppose there exists a mass $m$ global Maxwellian $M(v)$ such that
		\begin{align*}
			H(f_{\mrm{in}}|M) \le \frac{1}{4C + 2C^2 m},
		\end{align*}
		where $C = 2\pi\norm{(1+r^{-1})B}_{L^\infty_{r,\theta}}$. Then the initial value problem \eqref{1D_Boltzmann} admits a global solution $f\in L^\infty_{\mrm{loc}}([0,\infty),X)$, which is unique in $L^\infty([0,T],X)$ for any $T<\infty$, and we have the bound
		\begin{align*}
			\norm{f(t)}_{L^\infty_x L^1_v} \le\norm{f(t)}_X\le e^{C\sqrt{1+t}},\qquad t\in [0,\infty)
		\end{align*}
		where $C = C(B,\norm{f_{\mrm{in}}}_X, m)$.
		
		Furthermore, if $f_{\mrm{in}}\in W^{k,1}_\ell$ for some $k,\ell\ge 0$, then $f\in C([0,\infty),W^{k,1}_\ell)$.
	\end{theorem}

For large data, we get the following strong global solution theory, which is the main contribution of this work.

	\begin{theorem}\label{Thm2}
		Let $B$ be a collision kernel satisfying \textbf{(H1)} and \textbf{(H2)}, and take initial data $f\in X\cap L^2_1(\Omega^1_x\times\RR_v^3)$. Then the initial value problem \eqref{1D_Boltzmann} admits a global solution $f\in L^\infty_{\mrm{loc}}([0,\infty),X)$, which is unique in $L^\infty([0,T],X)$ for any $T$, and we have the  estimate
		\begin{align}
			\norm{f(t)}_{L^\infty_x L^1_v}\le\norm{f(t)}_X\le C e^{Ct} \qquad t\in [0,\infty)
		\end{align}
		when $\Omega_x^1 = \mathbb{T}^1$, and
		\begin{align}
			\norm{f(t)}_{L^\infty_x L^1_v}\le C \qquad t\in [0,\infty)
		\end{align}
		when $\Omega_x^1 = \RR^1$, as well as dissipation
		\begin{align}
			\norm{f(t)}_{L^\infty_x L^1_v}\to 0\qquad\hbox{as}\quad t\to\infty.
		\end{align}
		when $\Omega_x^1 = \RR^1$ and $f_{\mrm{in}}\in W^{1,1}(\RR_x^1\times\RR_v^3)$,
		where $C = C(\norm{f_{\mrm{in}}}_X,\norm{f_{\mrm{in}}}_{L^1(1+|v|^2)},B)$.

		Furthermore, if $f_{\mrm{in}}\in W^{k,1}_\ell$ for some $k,\ell\ge 0$, then $f\in C([0,\infty),W^{k,1}_\ell)$.
	\end{theorem}

	The proofs of Theorem \ref{Thm1} and Theorem \ref{Thm2} occupy Section \ref{Thm1_Proof_Section} and Section \ref{Thm2_Proof_Section} respectively.
	
	\section{Local Theory}
	
	Our local well-posedness result will follow from a bilinear estimate on the collision kernel $Q$, which we will prove by means of an angular averaging estimate. We first define the \emph{gain} and \emph{loss} terms, denoted $Q^+$ and $Q^-$ respectively, as
	
	\begin{align*}
		Q^+(g,f) = \int_{\RR^3_{v_*}\!\times\mathbb{S}^2_\sigma}
		B\left( |v-v_*|,\frac{v-v_*}{|v-v_*|}\cdot\sigma\right)
		g(v_*')f(v')\dd\sigma\dd v_*
	\end{align*}
	and
	\begin{align*}
		Q^-(g,f) = \int_{\RR^3_{v_*}\!\times\mathbb{S}^2_\sigma}
		B\left( |v-v_*|,\frac{v-v_*}{|v-v_*|}\cdot\sigma\right)
		g(v_*)f(v)\dd\sigma\dd v_*,
	\end{align*}
	such that $Q(g,f) = Q^+(g,f) - Q^-(g,f)$, where $v',v_*'$ are defined as in \eqref{post-collisional_velocities}, and $Q^\pm(g,f)\ge0$. The local well-posedness theory for \eqref{1D_Boltzmann} in the space $X$ will be proved using the following identity.

	\begin{lemma}\label{ContEstLemv2} Let $B$ be a collision kernel satisfying \textbf{(H1)}. For any $f,g\in X$, and for any $q\ge0$, we have
		\begin{equation}
			\| S_q Q^\pm(g,f) \|_{X}\le 8\pi\norm{\phi}_{L^1[0,\infty)}\| f\|_{X}\| g\|_X.
		\end{equation}
	\end{lemma}
	
	The proof of this lemma depends on the following angular averaging estimate. We note that the result for the full Grad cutoff case $B(r,\cos\theta)\in L^1(\sin\theta d\theta)$ is not available --
	this estimate is proved by converting an integral in $\sigma$ into an integral in the spatial variable, so we therefore require $B(r,\cos\theta)\in L^\infty_\theta$.

	\begin{proposition}[due to \cite{BiryukCraigPanferov}]\label{AAlemv2}  Let $B(r,\cos\theta)$ be a bounded collision kernel, and let $\widetilde B(r) = \sup_{|\xi|\le1}B(r,\xi)$. Then for all $q>0$, we have the bound
		
		\begin{align*}
			&\int_{\RR^3}\! Q^+(g,f)(x-qv_1,v)\,\dd v \\
			&\quad\le 		4\pi\int_{\RR_v^3}\!\int_{\RR_{v_*}^3}\!
			\int_{(q/2)(v_1 + v_{*,1} -|v-v_*|)}^{(q/2)(v_1 + v_{*,1} +|v-v_*|)}
			g(x-y,v_*)f(x-y,v)\frac{\widetilde B(|v-v_*|)}{q|v-v_*|}\dd y \dd v_*\dd v.
		\end{align*} 
		
	\end{proposition}
	
	\begin{proof}
		This proposition is found in \cite{BiryukCraigPanferov}, and the proof is similar to the outline in  \cite{BCPslides}. For completeness, we reproduce a full proof here. We expand the gain term $Q^+$ to obtain

	\begin{align}
		&\int_{\RR^3}\!Q^+(g,f)(x-qv_1,v)\,\dd v  	\nonumber\\
		&= \int_{\RR^3}\!\int_{\RR^3}\!\int_{\mathbb{S}^2}\! g(x-qv_1,v_*') f(x-qv_1,v')
			B\left (|v-v_*|,\sigma\cdot\frac{v-v_*}{|v-v_*|}\right) 
			\dd\sigma\dd v_* \dd v 		\nonumber\\
		&= \int_{\RR^3}\!\int_{\RR^3}\!\int_{\mathbb{S}^2}\! g(x-qv_1',v_*) f(x-qv_1',v)
			B\left(|v-v_*|,\sigma\cdot\frac{v-v_*}{|v-v_*|}\right) 
			\dd\sigma\dd v_* \dd v 			\nonumber\\
		&\le \int_{\RR^3}\!\int_{\RR^3}\!\int_{\mathbb{S}^2}\! g(x-qv_1',v_*) f(x-qv_1',v)
			\widetilde B(|v-v_*|) \dd\sigma\dd v_* \dd v
		\label{AAlemv2_1}
\end{align}

	where we have applied a classical pre-postcollisional change of variables in the second equality -- that is, we interchange the variables $(v',v_*')$ and $(v,v_*)$, which is an operation with unit Jacobian. We now wish to expand 
	$$
		v_1' = \frac{v_1 + v_{*,1}}{2} + \sigma\cdot e_1\frac{|v-v_*|}{2},
	$$	
	and we would like to integrate in $\sigma$. Setting $F(x,v,v_*) = g(x,v_*) f(x,v)$, we examine the integral \eqref{AAlemv2_1} in $\sigma$ and get

	\begin{align*}
		\int_{\mathbb{S}^2}\! F(x-qv_1',v,v_*)\dd\sigma
		&= \int_{\mathbb{S}^2}\! 
			F\left( x-q\left(\frac{v_1 + v_{*,1}}{2} +\sigma\cdot e_1\frac{|v-v_*|}{2}\right),
			v,v_* \right)
			\dd\sigma \\
		&= \int_0^{2\pi}\!\!\int_0^\pi\! 
			F\left( x-q\left(\frac{v_1 + v_{*,1}}{2} +\cos\theta\frac{|v-v_*|}{2}\right),
			v,v_* \right)\sin\theta
		\dd\theta\dd\varphi \\
		&= 2\pi\int_{-1}^1\!
			F\left( x-q\left(\frac{v_1 + v_{*,1}}{2} +z\frac{|v-v_*|}{2}\right),v,v_* \right)
			\dd z,
	\end{align*}
	where in the second equality we have used spherical coordinates $(\theta,\varphi)$ about the $e_1$ axis, setting $\cos\theta = \sigma\cdot e_1$. In the next equality, we have used the change of variables $z = \cos\theta$ with $\mrm{d}z = -\sin\theta\dd\theta$. We now use this to rewrite \eqref{AAlemv2_1} as
	\begin{align*}
		&
		2\pi\int_{\RR^3}\!\int_{\RR^3}\!
			\int_{-1}^1\!
			F\left( x-q\left(\frac{v_1 + v_{*,1}}{2} +z\frac{|v-v_*|}{2}\right),v,v_* \right)
		\widetilde B(|v-v_*|) \dd z\dd v_* \dd v \\
		&\qquad\qquad=
		4\pi\int_{\RR_v^3}\!\int_{\RR_{v_*}^3}\!
		\int_{(q/2)(v_1 + v_{*,1} -|v-v_*|)}^{(q/2)(v_1 + v_{*,1} +|v-v_*|)}
		g(x-y,v_*)f(x-y,v)\frac{\widetilde B(|v-v_*|)}{q|v-v_*|}\dd y \dd v_*\dd v
	\end{align*}
	where we have used the change of variables 
	\begin{align*}
		y = q\left(\frac{v_1 + v_{*,1}}{2} +z\frac{|v-v_*|}{2}\right).
	\end{align*}
	\end{proof}
	
	We can now prove Lemma \ref{ContEstLemv2}.
	
	\begin{proof}[Proof of Lemma \ref{ContEstLemv2}]
		We first note that 
		\begin{align*}
			\norm{S_q Q^\pm(g,f)}_X 
			&= \sup_{s\ge0}\norm{S_{q+s} Q^\pm(g,f)}_{L^\infty_x L^1_v} \\
			&\le\sup_{s\ge0}\norm{S_{s} Q^\pm(g,f)}_{L^\infty_x L^1_v} \\
			&=\sup_{s\ge0} \esssup_{x\in\Omega_x^1}\int_{\RR^3}\! Q^\pm(g,f)(x-sv_1,v)\dd v.
		\end{align*}
		We consider the two cases $Q^+$ and $ Q^-$ separately. We define $\Phi(u) = \int_{\mathbb{S}^2}B(u,\sigma)\dd \sigma$, and re-express $Q^-(g,f) = f(\Phi*g)$. Picking arbitrary $s$ and $x$, we estimate
		
		\begin{align*}
			\int_{\RR_v^3}\! Q^-(g,f)(x-sv_1,v)\dd v 
			&= \int_{\RR_v^3}\!\int_{\RR_{v_*}^3}\! g(x-sv_1,v_*)\Phi(v-v_*)f(x-sv_1,v) \dd v_*\dd v \\
			&\le\norm{\Phi* g}_{L^\infty_x L^\infty_v}\int_{\RR_v^3} 	f(x-sv_1,v) \dd v, \\
			&\le\norm{\Phi}_{L^\infty}\norm{g}_{L^\infty_x L^1_v} \norm{S_s f}_{L^\infty_x L^1_v} \\
			&\le 4\pi\norm{B}_{L^\infty_{r,\theta}}\norm{g}_X\norm{f}_X \\
			&\le 4\pi\norm{\phi}_{L^1[0,\infty)}\norm{g}_X\norm{f}_X,
		\end{align*}
		where the first and second inequalities hold for almost every $x$, and where in the last inequality we have used hypothesis \textbf{(H1)} to estimate
			\begin{align*}
			\sup_{r,\theta} B(r,\cos\theta)
				&=\sup_{r,\theta}\int_0^\infty 1_{0\le s\le r}\frac{B(r,\cos\theta)}{r}\dd s \\
				&\le\int_0^\infty\sup_{r,\theta}1_{0\le s\le r}\frac{B(r,\cos\theta)}{r}\dd s \\
				&\le\int_0^\infty\sup_{r,\theta} 1_{0\le s\le r}\phi(r)\dd s
			=\int_0^\infty\phi(s)\dd s = \norm{\phi}_{L^1([0,\infty))}.
		\end{align*}
		
		We now would like to show the same bound for the $Q^+(g,f)$ term. Choosing $s,x$ arbitrary, we apply the angular averaging estimate from Proposition \ref{AAlemv2} and get
		
		\begin{align*}
			&\int_{\RR^3}\! Q^+(g,f)(x-sv_1,v)\dd v \\
			&\le 4\pi\int_{\RR_v^3}\!\int_{\RR_{v_*}^3}\!
			\int_{(q/2)(v_1 + v_{*,1}-|v-v_*|)}^{(q/2)(v_1 +v_{*,1}+|v-v_*|)}
			g(x-y,v_*)f(x-y,v)\frac{\widetilde B(|v-v_*|)}{q|v-v_*|}\dd y \dd v_*\dd v.
			\\
			&\le
			4\pi\int_{\RR_v^3}\!\int_{\RR_{v_*}^3}\!
			\int_{q(v_1 -|v-v_*|)}^{q(v_1 +|v-v_*|)}
			g(x-y,v_*)f(x-y,v)\frac{\widetilde B(|v-v_*|)}{q|v-v_*|}\dd y \dd v_*\dd v
		\end{align*}
		where in the last step we have expanded the range of integration. By changing variables as $w = y-q v_1$, we can bound this as
		
		\begin{align*}
			&4\pi  \int_{\RR_v^3}\!\int_{\RR_{v_*}^3}\!
				\int_{- q|v-v_*|}^{q|v-v_*|}
						g(x-w-qv_1,v_*) f(x-w-q v_1,v) \frac{\widetilde B(|v-v_*|)}{q|v-v_*|}\dd w\dd v_*\dd v \\
			&\le  4\pi\int_{\RR_v^3}\!\int_{\RR_{v_*}^3}\!
				\int_{- q|v-v_*|}^{q|v-v_*|}
					g(x-w-qv_1,v_*) f(x-w-q v_1,v) \frac{\phi(|v-v_*|)}{q}\dd w \dd v_*\dd v \\
			&\le 4\pi \int_{\RR_v^3}\!\int_{\RR_{v_*}^3}\!
				\int_{- q|v-v_*|}^{q|v-v_*|}
					g(x-w-qv_1,v_*) f(x-w-q v_1,v)\frac{\phi(|w|/q)}{q}\dd w \dd v_*\dd v 
		\end{align*}
		where in the last inequality we have used that the integration bounds imply $|w|\le q|v-v_*|$ and the fact that $\phi$ is non-increasing by hypothesis \textbf{(H1)}. We can then bound this by
		
		\begin{align*}
			&
			4\pi \int_{\RR_v^3}\!\int_{\RR_{v_*}^3}\!
			\int_{\RR}
			g(x-w-qv_1,v_*) f(x-w-q v_1,v) \frac{\phi(|w|/q)}{q}\dd w \dd v_*\dd v 
			\\
			 &\le 4\pi\norm{g}_{L^\iy_xL^1_v}\int_{\RR_v^3}\!\int_{\RR}
			f(x-w-q v_1,v)\frac{\phi(|w|/q)}{q}\dd w \dd v \\
			 & \le 4\pi \norm{g}_{L^\iy_xL^1_v}\norm{f}_X 
			\int_\RR\frac{\phi(|w|/q)}{q}\dd w						\\
			 & =8\pi \norm{\phi}_{L^1(\RR)}\norm{g}_X\norm{f}_X ,
		\end{align*}
		where the two inequalities hold for all $x$, due to integration in $w$. This concludes the proof.
	\end{proof}
	
	These estimates now suffice to prove a local well-posedness theory in the space $X$.
	
	\begin{lemma} \label{local_lemma}
		Let $0\le f_{\mrm{in}}\in X$. Then the problem \eqref{1D_Boltzmann} admits a  solution $f\in L^\infty_{\mrm{loc}}([0,T_*),X)$ for some maximal time $T_*\in (0,\infty]$, which is unique in $L^\infty([0,T],X)$ for any $T<T_*$. If $T_*<\infty$, then $\norm{f(t)}_X$ is unbounded over $t\in [0,T_*)$.
	\end{lemma}
	\begin{proof}
		The proof follows by a standard contraction mapping argument in $L^\infty([0,T],X)$. By construction of  the space $X$, we have 
		\begin{align*}
			(t\mapsto S_t f_{\mrm{in}}) \in L^\infty([0,T],X)
		\end{align*}
	since $\norm{S_t f_{\mrm{in}}}_X$ is non-increasing in $t$. By Lemma \ref{ContEstLemv2} we have that the operator
	\begin{align*}
		\mathcal R:L^\infty([0,T],X)\to L^\infty([0,T],X),\quad
		\mathcal{R}f(t) = S_t f_{\mrm{in}} + \int_0^t S_{t-s} Q(f,f)(s)\dd s
	\end{align*}
	whose fixed points solve \eqref{1D_Boltzmann} is continuous, and by standard bilinear identities we get the estimate
	\begin{align*}
		\norm{\ca R f - \ca R g}_{L^\infty_T\! X}
		\le C(B) T \left(\norm{f}_{L^\infty_T\! X}+\norm{g}_{L^\infty_T\! X}\right)\norm{f-g}_{L^\infty_T\! X}.
	\end{align*}
	We define the first iterate $f_0 = S_t f_{\mrm{in}}$, and choose $R\ge 2\norm{f_{\mrm{in}}}_{ X}$. If $\norm{f}_{L^\infty_T X}\le R$, then the bound
	\begin{align*}
		\norm{\ca R f }_{L^\infty_T X} \le \norm{f_0}_{L^\infty_T X} + 
		 C T\norm{f}_{L^\infty_T X}^2
		&\le \norm{f_0}_{L^\infty_T X} +  C T R^2  
		\\
		&\le R
	\end{align*}
	holds for $T>0$ sufficiently small. If we also have $\norm{g}_{L^\infty_T X}\le R$ then the bound
	\begin{align*}
		\norm{\ca R f - \ca R g}_{L^\infty_T X} \le 2CT R \norm{f-g}_{L^\infty_T X}\le \frac12 \norm{f-g}_{L^\infty_T X}
	\end{align*}
	holds for $T\le 1/4CR$, so that $\ca R$ is a contraction on the ball of radius $R$ in $L^\infty_T X$ for $T\le C' R^{-1}$, and \eqref{1D_Boltzmann} has unique solutions in this ball. 
	
	We now prove uniqueness on larger intervals. If $f,g\in L^\infty([0,\overline T],X)$ were two different solutions for some $\overline T>0$, then we define the minimal time at which they differ as
	\begin{align*}
		0<T\le \tau = \inf\{t\in [0,\overline T]\,:\,f(t)\neq g(t)\}.
	\end{align*}
	Choosing $R\ge\max(\norm{f}_{L^\infty_T X},\norm{g}_{L^\infty_T X})$ and taking $f(\tau-\epsilon) = g(\tau-\epsilon)$ as initial data for $\epsilon$ sufficiently small and satisfying $2 \epsilon< C' R^{-1}$, we can apply a contraction mapping argument on the time interval $[\tau-\epsilon,\tau+\epsilon]$, which gives that $f(\tau+\epsilon) = g(\tau+\epsilon)$. By contradiction, this proves uniqueness on $[0,\overline{T}]$.
	
	Similarly, if $f\in L^\infty([0,T_*),X)$ is a solution and $T_*<\infty$, then choosing $R\ge\norm{f}_{L^\infty([0,T_*), X)}$, choosing $\epsilon$ such that $2\epsilon< C' R^{-1}$ and taking as initial data $f(T_*-\epsilon)$, we can extend $f$ to a solution on $[0,T_*+\epsilon]$, so that the time $T_*$ is not maximal.
	\end{proof}

	\section{Higher Order Estimates}
	
	We would now like to prove a propagation of regularity result for the solutions $f\in L^\infty_{\mrm{loc}}([0,T_*),X)$ provided by the local theory. By analogy with the Beale-Kato-Majda criterion for 3D Navier-Stokes \cite{BealeKatoMajda84}, we will show that the norm $X$ alone controls all higher order norms $W^{k,1}_\ell$, and therefore a bound $f\in L^\infty([0,T],X)$ is enough to show boundedness of the $W^{k,1}_\ell$ norms in time $t\in [0,T]$.
	
	Our strategy can be summarized with the following diagram,
	
	\begin{equation}
	\label{section_3_diagram}
		\begin{tikzcd}
		X \arrow[r] \arrow[dr] &  L_\ell^1 \\
		&
		W_0^{1,1} \arrow[r] 
		& 
		W_\ell^{1,1} \arrow[r]
		&
		W^{2,1}_\ell \arrow[r]
		&
		\cdots \arrow[r]
		&
		W_\ell^{k,1}
	\end{tikzcd}
	\end{equation}
	where $\ca Y\to \ca Z$ represents the statement that for any solution $f$ to \eqref{1D_Boltzmann}, we have
	\begin{align*}
		f\in L^1([0,T),\ca Y),\quad f|_{t=0}\in\ca Z\quad \implies\quad f\in C_{\mrm{b}}([0,T),\ca Z)
	\end{align*}
	for any $T<\infty$. We note that this property is transitive. Since free transport is a strongly continuous semigroup in $W^{k,1}_\ell$ for all $k,\ell\ge0$, the proof of the diagram \eqref{section_3_diagram} will reduce to the following estimates on the collisional terms $Q^\pm$, and to standard Gronwall estimates.
	
	\begin{lemma}\label{XtoL^1_ell_Lemma}
		For any $f,g\in X\cap L^1_\ell(\Omega_x^1\times\RR_v^3)$, we have the bound
		\begin{align*}
			\norm{Q^\pm(g,f)}_{L^1_\ell}
			\le C_\ell\norm{\Phi}_{L^\infty_v}
			\left(
			\norm{g}_{L^1_\ell}\norm{f}_X + \norm{g}_X \norm{f}_{L^1_\ell}
			\right).
		\end{align*}
	\end{lemma}
	\begin{proof}
		The proof is elementary, and is included for completeness. We first get a bound in the $v$ variable, and can then apply a suitable Hölder estimate in $x$. Since $Q^-(g,f) = (\Phi*g)f$, we can bound 
		\begin{align*}
			\norm{Q^-(g,f)m_\ell}_{L^1_v}
			= \norm{(\Phi*g)f m_\ell}_{L^1_v} 
			\le\norm{\Phi}_{L^\infty_v}\norm{g}_{L^1_v}\norm{fm_\ell}_{L^1_v}
		\end{align*}
		by standard convolution estimates. We can similarly bound the gain term as
		\begin{align*}
			\norm{Q^+(g,f)m_\ell}_{L^1_v}
			&=\int_{\RR_v^3\times\RR_{v_*}^3\times\SSS_\sigma^2}\!
			Bg_*'f' m_\ell(x,v)\dd\sigma\dd v_*\dd v \\
			&=\int_{\RR_v^3\times\RR_{v_*}^3\times\SSS_\sigma^2}\!
			Bg_*f m_\ell(x,v')\dd\sigma\dd v_*\dd v  \\
			&\le
			\int_{\RR_v^3\times\RR_{v_*}^3\times\SSS_\sigma^2}\!
			Bg_*f 2^{\ell-1}(m_\ell(x,v)+m_\ell(x,v_*))\dd\sigma\dd v_*\dd v,
		\end{align*}
		where we have used a pre-post collisional change of variables. Applying the bound on $Q^-$ gives the estimates
		\begin{align}
			\label{moment_control_in_v}
			\norm{Q^\pm(g,f)m_\ell}_{L^1_v}
			\le C_\ell\norm{\Phi}_{L^\infty_v}
			\left(
			\norm{g m_\ell}_{L^1_v}\norm{f}_{L^1_v} + \norm{g}_{L^1_v}\norm{f m_\ell}_{L^1_v}
			\right).
		\end{align}
		
		Integrating in $x$ gives the bound
		\begin{align*}
			\norm{Q^\pm(g,f)}_{L^1_\ell} 
			&\le
			C_\ell\norm{\Phi}_{L^\infty_v}
			\left(
			\norm{g m_\ell}_{L^1_{x,v}}\norm{f}_{L^\infty_x L^1_v}
			+
			\norm{g}_{L^\infty_x L^1_v}\norm{fm_\ell}_{L^1_{x,v}}
			\right) \\
			&\le C_\ell\norm{\Phi}_{L^\infty_v}
			\left(
			\norm{g}_{L^1_\ell}\norm{f}_X + \norm{g}_X \norm{f}_{L^1_\ell}
			\right),
		\end{align*}
		where we have used the bound $\norm{g}_{L^\infty_x L^1_v}\le\norm{g}_X$, and this concludes the proof.
	\end{proof}
	
	\begin{lemma} \label{Lemma3.2}
		For any $f,g\in W^{1,1}(\Omega_x^1\times\RR_v^3)$, we have the bound
		\begin{align*}
			\norm{Q^\pm(g,f)}_{W^{1,1}}
			\le
			\norm{\Phi}_{L^\infty_v}
			\left(
			\norm{g}_{W^{1,1}}\norm{f}_{X}
			+
			\norm{g}_{X}\norm{f}_{W^{1,1}}
			\right).
		\end{align*}
	\end{lemma}
	\begin{proof}
		Using Lemma \ref{XtoL^1_ell_Lemma} we can control $\norm{Q^\pm(g,f)}_{L^1}$ so it remains to control $\norm{\na_{x,v}Q^\pm(g,f)}_{L^1}$. Using the classical identity  \begin{align*}
			\na_{x,v}Q^\pm(g,f) = Q^\pm(\na_{x,v}g,f) + Q^\pm(g,\na_{x,v}f)
		\end{align*}
		and the trivial convolutional bound $\norm{Q^\pm(g,f)}_{L^1_v}\le\norm{\Phi}_{L^\infty_v}\norm{g}_{L^1_v}\norm{f}_{L^1_v}$, we estimate
		\begin{align*}
			\norm{\na_{x,v}Q^\pm(g,f)}_{L^1_{x,v}}
			&\le
			\norm{\Phi}_{L^\infty_v}\int_{\Omega_x^1}\! 
			\left(
			\norm{\na_{x,v}g}_{L^1_v}\norm{f}_{L^1_v} 
			+
			\norm{g}_{L^1_v}\norm{\na_{x,v}f}_{L^1_v}
			\right)
			\dd x \\
			&\le
			\norm{\Phi}_{L^\infty_v}
			\left(
			\norm{\na_{x,v}g}_{L^1_{x,v}}\norm{f}_{L^\infty_x L^1_v}
			+
			\norm{g}_{L^\infty_x L^1_v}\norm{\na_{x,v}f}_{L^1_{x,v}}
			\right) \\
			&\le \norm{\Phi}_{L^\infty_v}
			\left(
			\norm{g}_{W^{1,1}}\norm{f}_{X}
			+
			\norm{g}_{X}\norm{f}_{W^{1,1}}
			\right),
		\end{align*}
		which concludes the proof.
	\end{proof}
	
	\begin{lemma} \label{Lemma3.3}
		For any $f,g\in W^{1,1}_\ell(\Omega_x^1\times\RR_v^3)$, we have the bound
		\begin{align*}
			\norm{Q^\pm(g,f)}_{W^{1,1}_\ell}
			\le
			C_\ell\norm{\Phi}_{L^\infty_v}
			\left(
			\norm{g}_{W^{1,1}_\ell}\norm{f}_{W^{1,1}}
			+
			\norm{g}_{W^{1,1}}\norm{f}_{W^{1,1}_\ell}
			\right).
		\end{align*}
	\end{lemma}
	\begin{proof}
		The proof follows classical ideas, and is very similar to Lemma 5.16 of \cite{GualdaniMischlerMouhot17}, with finer control of moments. The bound on $\norm{Q^\pm(g,f)}_{L^1_\ell}$ follows from Lemma \ref{XtoL^1_ell_Lemma}. It then suffices to estimate
		\begin{align*}
			\norm{\na_{x,v}Q^\pm(g,f)}_{L^1_\ell}
			&\le
			\norm{Q^\pm(\na_{x,v}g,f)}_{L^1_\ell} 
			+
			\norm{Q^\pm(g,\na_{x,v}f)}_{L^1_\ell} \\
			&\le
			C_\ell\norm{\Phi}_{L^\infty_v}
			\int_{\Omega_x^1}\!
			\Big(
			\norm{\na_{x,v}g\,m_\ell}_{L^1_v}\norm{f}_{L^1_v}
			+
			\norm{\na_{x,v}g}_{L^1_v}\norm{fm_\ell}_{L^1_v} \\
			&\qquad\qquad\qquad\qquad
			+
			\norm{gm_\ell}_{L^1_v}\norm{\na_{x,v}f}_{L^1_v}
			+
			\norm{g}_{L^1_v}\norm{\na_{x,v}f\,m_\ell}_{L^1_v}
			\Big)
			\dd x \\
			&\le
			C_\ell\norm{\Phi}_{L^\infty_v}
			\Big(
			\norm{\na_{x,v}g\,m_\ell}_{L^1_{x,v}}\norm{f}_{L^\infty_xL^1_v}
			+
			\norm{\na_{x,v}g}_{L^1_{x,v}}\norm{fm_\ell}_{L^\infty_x L^1_v} \\
			&\qquad\qquad\qquad\ \ 
			+
			\norm{gm_\ell}_{L^\infty_xL^1_v}\norm{\na_{x,v}f}_{L^1_{x,v}} 
			+
			\norm{g}_{L^\infty_x L^1_v}\norm{\na_{x,v}f\,m_\ell}_{L^1_{x,v}}
			\Big) \\
			&\le 
			C_\ell'\norm{\Phi}_{L^\infty_v}
			\left(
			\norm{gm_\ell}_{W^{1,1}}\norm{f}_{W^{1,1}}
			+
			\norm{g}_{W^{1,1}}\norm{f m_\ell}_{W^{1,1}}
			\right)
		\end{align*}
		where we have used the Sobolev embedding $W^{1,1}(\Omega_x^1)\subset L^\infty(\Omega_x^1)$ and collected terms in the final inequality.
	\end{proof}
	
	\begin{lemma} \label{Lemma3.4}
		For any $f,g\in W^{k,1}_\ell(\Omega_x^1\times\RR_v^3)$ for $k\ge2$, we have the bound
		\begin{align*}
			\norm{Q^\pm(g,f)}_{W^{k,1}_\ell}
			\le
			C_{k,\ell}\norm{\Phi}_{L^\infty_v}
			\left(
			\norm{g}_{W^{k,1}_\ell}\norm{f}_{W^{k-1,1}_\ell}
			+
			\norm{g}_{W^{k-1,1}_\ell}\norm{f}_{W^{k,1}_\ell}
			\right).
		\end{align*}
	\end{lemma}
	\begin{proof}
		The proof follows by classical ideas, and is similar to the proof of Theorem 2.1 of \cite{MouhotVillani04} with finer control on derivatives -- we use the Leibniz rule to expand 
		\begin{align*}
			\norm{\na_{x,v}^k Q^\pm(g,f)}_{L^1_\ell}
			&\le
			2^k\sum_{j=0}^k \norm{Q^\pm(\na_{x,v}^j g,\na_{x,v}^{k-j}f)}_{L^1_\ell} \\
			&\le
			\sum_{j=0}^k C_{k,\ell}	\norm{\Phi}_{L^\infty_v}
			\int_{\Omega_x^1}\!
			\norm{\na_{x,v}^j g\, m_\ell}_{L^1_v} \norm{\na_{x,v}^{k-j}f\, m_\ell}_{L^1_v}
			\dd x
		\end{align*}
		where we have applied \eqref{moment_control_in_v},
		and we conclude by standard Sobolev embeddings.
	\end{proof}
	
	\begin{proposition}\label{regularity_proposition}
		Take initial data $0\le f_{\mrm{in}}\in X$, and let $f\in L^\infty_{\mrm{loc}}([0,T_*),X)$ be the unique solution to \eqref{1D_Boltzmann} from Lemma \ref{local_lemma}. If $f_{\mrm{in}}\in W^{k,1}_\ell$ for some $k,\ell\ge 0$, then $f\in C([0,T_*),W^{k,1}_\ell)$. 
		
		If $f_{\mrm{in}}\in L^1(1+|v|^2)$, then $f$ satisfies the conservation laws \eqref{conservation_laws}. Furthermore, if $f_{\mrm{in}}\in L^1(1+|v|^2)\cap L\log L$, then $f$ also satisfies the entropy inequality \eqref{entropy_inequality}.
	\end{proposition}
	\begin{proof}
		The proof that $f\in C([0,T_*),W^{k,1}_\ell)$ reduces to proving the diagram \eqref{section_3_diagram}. Arguing abstractly, we suppose given Banach spaces $\ca Y,\ca Z$ such that $S_t$ is a strongly continuous semigroup in $\ca Z$, with $\norm{S_t f_{\mrm{in}}}_{\ca Z}\le C_{\ca Z}(t)\norm{f_\mrm{in}}_{\ca Z}$ for $C_{\ca Z}\in L^\infty_{\mrm{loc}}([0,\infty))$, and such that the estimate
		\begin{align}
			\label{Q_Y_Z_bound}
			\norm{Q(f,f)}_{\ca Z}\le C_{\ca Y,\ca Z}\norm{f}_{\ca Y}\norm{f}_{\ca Z}
		\end{align}
		holds. We can then take the $\ca Z$ norm of the Duhamel formulation \eqref{Duhamel_Boltzmann} to get that
		\begin{align*}
			\norm{f(t)}_{\ca Z} \le C_{\ca Z}(t)\norm{f_{\mrm{in}}}_{\ca Z}
			+
			\int_0^t\! C_{\ca Z}(t-s)C_{\ca Y,\ca Z}\norm{f(s)}_{\ca Y}\norm{f(s)}_{\ca Z}\dd s.
		\end{align*}
		By a Gronwall estimate, we then immediately get that if $f\in L^1([0,T),\ca Y)$ for some $T<\infty$ and $f_{\mrm{in}}\in\ca Z$, then we have $f\in L^\infty([0,T),\ca Z)$. Since we have that $S_tf_{\mrm{in}}\in C([0,\infty),\ca Z)$ by strong continuity, we can use the bound \eqref{Q_Y_Z_bound} to immediately deduce continuity $f\in C_{\mrm{b}}([0,T),\ca Z)$ from the Duhamel formulation \eqref{Duhamel_Boltzmann}.
		
		We now note that our assumptions on $\ca Y$ and $\ca Z$ hold for every arrow $\ca Y\to\ca Z$ in the diagram \eqref{section_3_diagram}. It is easy to check strong continuity of $S_t$ in $W^{k,1}_\ell$, and furthermore we have the estimate $\norm{S_t f_{\mrm{in}}}_{W^{k,1}_\ell}\le C_{k,\ell}(1+t)^{k+\ell}$. The bound \eqref{Q_Y_Z_bound} was proved in Lemmas \ref{XtoL^1_ell_Lemma}, \ref{Lemma3.2}, \ref{Lemma3.3}, and \ref{Lemma3.4}. 
		
		By assumption in the proposition, we have $f\in L^\infty_{\mrm{loc}}([0,T_*),X)$, which implies that $f\in L^1([0,T),X)$ for any finite $T<T_*$, and we have that $f_{\mrm{in}}\in W^{k,1}_\ell$. But using transitivity of the diagram \eqref{section_3_diagram}, we have shown that this implies $f\in C_{\mrm{b}}([0,T),W^{k,1}_\ell)$ for any such $T$, which proves the claim.

		To prove the next claim, we consider $f_{\mrm{in}}\in X\cap L^1(1+|v|^2)$, which therefore implies $f\in L^\infty([0,T],X\cap L^1(1+|v|^2))$ for $T<T_*$,
		which then implies that $Q(f,f)\in L^1(1+|v|^2)$ by Lemma \ref{XtoL^1_ell_Lemma}. We then immediately get that the conservation laws \eqref{conservation_laws} hold by the cancellation properties \eqref{collision_kernel_cancellation} of $Q(f,f)$.
		
		To prove the entropy inequality when $f\in X\cap L^1(1+|v|^2)\cap L\log L$, we note that $f$ satisfies a weak-strong uniqueness criterion proved by Lions \cite{Lions94}. Since 
		\begin{align*}
			\norm{\Phi*f(t)}_{L^\infty_{x,v}}\le\norm{\Phi}_{L^\infty}\norm{f}_{L^\infty_x L^1_v}
			\le\norm{\Phi}_{L^\infty}\norm{f}_X
		\end{align*}
		and since $\norm{|v| f(t)}_{L^1_{x,v}}$ is uniformly bounded for $t\in [0,T]$, we see that the conditions of Theorem IV.1 and V.1 of \cite{Lions94} are satisfied. Therefore, there exists a renormalized solution $g\in C([0,\infty),L^1_{x,v})$ with initial data $f_{\mrm{in}}$ which coincides with $f$ on the time interval $[0,T]$. By construction, this $g$ satisfies the entropy inequality \eqref{entropy_inequality} for all time $t\in[0,\infty)$, which proves the result for $f$.
		
	\end{proof}

	\section{Global Solutions Near Maxwellians} \label{Thm1_Proof_Section}
	
	In this section we present a proof of Theorem \ref{Thm1}. This result is found in \cite{BiryukCraigPanferov} and is not new, but we hope that a complete proof can be useful for the literature. Our method is to use the angular averaging estimate from Proposition \ref{AAlemv2} to get stronger control on the growth of $\norm{f(t)}_X$. We do this by bounding the term $\norm{S_{t-s}Q^+(f,f)(s)}_X$ using entropy estimates. By combining these estimates with the bilinear estimate $\norm{S_{t-s}Q^+(f,f)(s)}_X\le C\norm{f(s)}_X^2$, we will be able to close an estimate for the norm $\norm{f(t)}_X$ and show that it remains bounded for finite time. Using the blowup criterion from Lemma \ref{local_lemma}, this will prove that our solutions are global in time.
	
		We now fix our assumptions for this section. We take initial data $f_{\mrm{in}}\in X\cap L^1(1+|v|^2)\cap L\log L(\TT^1_x\times\RR^3_v)$, and let $f\in L^\infty_{\mrm{loc}}([0,T_*),X)$ be the unique solution provided by Lemma \ref{local_lemma}. We define the mass
	\begin{align*}
		m = \int_{\TT_x^1}\!\int_{\RR_v^3}\! f(t,x,v)\dd v \dd x
	\end{align*}
	which is constant in time. Using Proposition \ref{AAlemv2}, we can bound
	\begin{align}
		&\norm{S_q Q^+(f,f)}_X \nonumber
		\\
		&\le
		4\pi
		\sup_{s\ge q}\esssup_{x\in\TT^1} 
		\int_{\RR_v^3}\!\int_{\RR_{v_*}^3}\!
		\int_{(s/2)(v_1 + v_{*,1} -|v-v_*|)}^{(s/2)(v_1 + v_{*,1} +|v-v_*|)}
		F(x-y,v,v_*)\frac{\widetilde B(|v-v_*|)}{s|v-v_*|}\dd y \dd v_*\dd v  \nonumber
		\\
		&\le
		4\pi \sup_{s\ge q}\esssup_{x\in\TT^1} 
		\int_{\RR_v^3}\!\int_{\RR_{v_*}^3}\!
		\int_{\TT^1}\lceil s|v-v_*|\rceil
		F(y,v,v_*)\frac{\widetilde B(|v-v_*|)}{s|v-v_*|}\dd y \dd v_*\dd v \nonumber
		\\
		&=
		4\pi
		\sup_{s\ge q} \int_{\RR_v^3}\!\int_{\RR_{v_*}^3}\!
		\int_{\TT_x^1}
		f(x,v)f(x,v_*)\widetilde B(|v-v_*|)\frac{\lceil s|v-v_*|\rceil}{s|v-v_*|}\dd x \dd v_*\dd v  \nonumber
		\\
		&\le
		4\pi
		\sup_{s\ge q} \int_{\RR_v^3}\!\int_{\RR_{v_*}^3}\!
		\int_{\TT_x^1}
		f(x,v)f(x,v_*)\widetilde B(|v-v_*|)
			\left(
				1 + \frac{1}{s|v-v_*|}
			\right)
		\dd x \dd v_*\dd v  \nonumber
		\\
		&\le
		4\pi
		\int_{\RR_v^3}\!\int_{\RR_{v_*}^3}\!
		\int_{\TT_x^1}
		f(x,v)f(x,v_*)\widetilde B(|v-v_*|)
		\left(
		1 + \frac{1}{q|v-v_*|}
		\right)
		\dd x \dd v_*\dd v,
		\label{second_bound}
	\end{align}
	where we have set $F(x,v,v_*) = f(x,v)f(x,v_*)$, and where $\lceil\cdot\rceil$ denotes the ceiling function. Using hypothesis \textbf{(H1)}, we can bound
	\begin{align*}
		\widetilde B(|v-v_*|)
		\left(
		1 + \frac{1}{q|v-v_*|}
		\right)
		\le
		\norm{(1+r^{-1})B(r,\cos\theta)}_{L^\infty_{r,\theta}}
		\left(
		1 + \frac{1}{q}
		\right),
	\end{align*}
	which allows us to bound \eqref{second_bound} by
	\begin{align}
		&4\pi \norm{(1+r^{-1})B}_{L^\infty_{r,\theta}}
		\left(
			1 + \frac{1}{q}
		\right)
		\int_{\RR_v^3}\!\int_{\RR_{v_*}^3}\!\int_{\TT_x^1}	f(x,v)f(x,v_*) \dd x \dd v_*\dd v \nonumber
		\\
		&=
		4\pi \norm{(1+r^{-1})B}_{L^\infty_{r,\theta}}
		\left(
		1 + \frac{1}{q}
		\right)
		\int_{\TT_x^1}\!\rho(x)^2\dd x,
		\label{third_bound}
	\end{align}
	where we define $\rho(x) = \int_{\RR_v^3}\! f(x,v)\dd v$. We can then estimate this using the following lemma.
	
	\begin{lemma}\label{rho_L2_lemma}
		Suppose we have $0\le f\in X$ on $\TT_x^1\times\RR_v^3$ with mass $m$ and density $\rho$, and let $M(v)$ be a Maxwellian of mass $m$. Take arbitrary $\epsilon>0$. Then we have the bound
		\begin{align*}
			&\int_{\TT^1}\!\rho(x)^2\dd x 
			\le
			\left(
				\epsilon m + H(f|M) + \sqrt{\frac{m}{2}H(f|M)}
			\right)
			\max\left(
				\frac{m}{\epsilon},\ 
				\frac{\norm{f}_X}{\log^+(m^{-1}\norm{f}_X)+\epsilon}
			\right).
		\end{align*}
	\end{lemma}
	\begin{proof}
		We first estimate
			\begin{align}
			\int_{\TT^1}\!\rho(x)^2\dd x
			&\le
			\int_{\TT^1}\! \rho(x)\left[ \epsilon + \log^+\frac{\rho(x)}{m}\right]\mathrm{d} x \ 
				\norm{\frac{\rho}{\epsilon + \log^+(m^{-1}\rho)}}_{L^\infty(\TT^1)} \nonumber
			\\
			&= \left(
				m\epsilon + H^+(\rho|m)
			\right)
			\norm{\frac{\rho}{\epsilon + \log^+(m^{-1}\rho)}}_{L^\infty(\TT^1)} \nonumber
			\\
			\begin{split}
			&\le
			\left(
				m\epsilon + H(\rho|m) + H^-(\rho|m)
			\right) 
			\\
			&\qquad\qquad\qquad\qquad\cdot
			\max\left(
				\frac{m}{\epsilon},\ 
				\frac{\norm{\rho}_{L^\infty}}{\epsilon+\log^+(m^{-1}\norm{\rho}_{L^\infty})}
			\right)
			\label{bound_on_rho_L2}
			\end{split}
		\end{align}
	where we have defined
	\begin{align*}
		H^\pm(\rho|m) = \int_{\TT^1}\!\rho(x)\log^\pm\frac{\rho(x)}{m}\dd x,
	\end{align*}
	and where in the second inequality we have used that the function
	\begin{align}\label{nondec_func_of_phi}
		\varphi\mapsto
		\max\left(
			\frac{m}{\epsilon},\ 
			\frac{\varphi}{\log^+(m^{-1}\varphi)+\epsilon}
		\right),
		\qquad m\le \varphi<\infty
	\end{align}
	is non-decreasing and therefore commutes with suprema. Since $m = \norm{\rho}_{L^1(\TT^1)}\le\norm{\rho}_{L^\infty(\TT^1)}$, the bound \eqref{bound_on_rho_L2} is justified. We can then control the negative part of the relative entropy as
	\begin{align*}
		H^-(\rho|m)
		&=
		\int_{\TT^1}\!\rho\log^-\frac{\rho}{m}\dd x
		\\
		&\le
		m\int \left( 1- \frac{\rho}{m}\right)_+ \dd x
		\le
		m\sqrt{\frac12 H(m^{-1}\rho)}
		=
		\sqrt{\frac{m}{2} H(\rho|m)}
	\end{align*}
	where we define $(\cdot)_+ = \max(\cdot,0)$ and where we have used Pinsker's inequality in the last inequality. Applying this bound to \eqref{bound_on_rho_L2} gives us the estimate
	\begin{align}
		&\int_{\TT^1}\!\rho(x)^2\dd x \nonumber
		\\
		&\le
		\left(
			m\epsilon + H(\rho|m) + \sqrt{\frac{m}{2}H(\rho|m)}
		\right)
		\max\left(
			\frac{m}{\epsilon},\ 
			\frac{\norm{\rho}_{L^\infty}}{\log^+(m^{-1}\norm{\rho}_{L^\infty})+\epsilon}
		\right).
		\label{second_bound_on_rho_L2}
	\end{align}
	We now apply the chain rule for relative entropy, to get the bound
	\begin{align*}
		H(f|M)
		&=
		\int_{\TT_x^1\times\RR_v^3}\! f(x,v)\log\frac{f(x,v)}{M(v)}\dd v\dd x \\
		&=
		\int_{\TT_x^1}\!\rho(x)\int_{\RR_v^3}\!\frac{f(x,v)}{\rho(x)}
		\left[
			\log\frac{f(x,v) m}{\rho(x) M(v)} + \log\frac{\rho(x)}{m}
		\right] \dd v\dd x 
		\\
		&=
		\int_{\TT_x^1}\!\rho(x)\int_{\RR_v^3}\! \rho(x)^{-1}f(x,v)
			\log\frac{\rho(x)^{-1}f(x,v) }{m^{-1} M(v)}
			\dd v\dd x
			+
		\int_{\TT_x^1}\!\rho(x)\log\frac{\rho(x)}{m}\dd x \\
		&\ge
		\int_{\TT_x^1}\! \rho(x)\log\frac{\rho(x)}{m}\dd x = H(\rho|m).
	\end{align*}
	Applying this to \eqref{second_bound_on_rho_L2}, we can bound
	\begin{align*}
		&\int_{\TT^1}\!\rho(x)^2\dd x 
		\\
		&\le
		\left(
			\epsilon m + H(f|M) + \sqrt{\frac{m}{2}H(f|M)}
		\right)
		\max\left(
			\frac{m}{\epsilon},\ 
			\frac{\norm{\rho}_{L^\infty}}{\log^+(m^{-1}\norm{\rho}_{L^\infty})+\epsilon}
		\right) 
		\\
		&\le
		\left(
			\epsilon m + H(f|M) + \sqrt{\frac{m}{2}H(f|M)}
		\right)
		\max\left(
			\frac{m}{\epsilon},\ 
			\frac{\norm{f}_{X}}{\log^+(m^{-1}\norm{f}_{X})+\epsilon}
		\right)
	\end{align*}
	where we have used the monotonicity of \eqref{nondec_func_of_phi} and the fact that $\norm{\rho}_{L^\infty}\le\norm{f}_X$ in the final inequality, and this concludes the proof.
	\end{proof}

	\begin{lemma} \label{small_entropy_lemma}
		Let $f$ and $M$ be as in Lemma \ref{rho_L2_lemma}, and suppose 
		\begin{align*}
			H(f|M)\le  \frac{1}{4C + 2C^2 m}
		\end{align*}
		where $C = 2\pi\norm{(1+r^{-1})B}_{L^\infty_{r,\theta}}$. Then there exists some $\alpha<1$ and $\epsilon>0$, such that we have the bound
		\begin{align*}
			\norm{S_q Q^+(f,f)}_X\le\alpha\left(1+\frac1q\right) 
			\max\left(
			\frac{m}{\epsilon},\ 
			\frac{\norm{f}_{X}}{\log^+(m^{-1}\norm{f}_{X})+\epsilon}
			\right),
		\end{align*}
	 where $\alpha,\epsilon$ depend on $\norm{(1+r^{-1})B}_{L^\infty_{r,\theta}}$ and $m$.
	\end{lemma}
	\begin{proof}
		By applying Lemma \ref{rho_L2_lemma} to the bound \eqref{third_bound}, we get
		 \begin{align*}
		 	&\norm{S_q Q^+(f,f)}_X \nonumber
		 	 \\
		 	&\qquad\le
		 	4\pi\norm{(1+r^{-1})B}_{L^\infty_{r,\theta}}
		 		\left(
		 			\epsilon m + H(f|M) + \sqrt{\frac{m}{2}H(f|M)}
		 		\right)
		 	\\
		 	&\qquad\qquad\qquad\qquad\qquad\qquad\qquad\cdot
		 	\left(1 + \frac1q\right)
		 	\max\left(
		 		\frac{m}{\epsilon},\ 
		 		\frac{\norm{f}_X}{\log(m^{-1}\norm{f}_X)+\epsilon}
		 	\right) \nonumber \\
		 	&\qquad=
		 	2C
			 	\left(
				 	\epsilon m + H(f|M) + \sqrt{\frac{m}{2}H(f|M)}
		 	\right)
			 	\\
		 	&\qquad\qquad\qquad\qquad\qquad\qquad\qquad\cdot
		 	\left(1 + \frac1q\right)
			\max\left(
			 	\frac{m}{\epsilon},\ 
		 		\frac{\norm{f}_X}{\log(m^{-1}\norm{f}_X)+\epsilon}
		 	\right) \nonumber 
		 \end{align*}
	 so it suffices to show that
	\begin{align}\label{bound_by_K(m,C)}
		2C 
		\left(
			H(f|M)+\sqrt{\frac{m}{2}H(f|M)}
		\right) 
		\le K(m,C)<1
	\end{align}
	for a suitable $K(m,C)$, since we could then set
	\begin{align*}
		\epsilon = \frac{1 - K(m,C)}{4Cm}
	\end{align*}
	and let $\alpha = K(m,C) + \epsilon m<1$, and this would prove the claim. We define the function
	\begin{align*}
		\varphi(h) = h + \sqrt{\frac{m}{2} h},
	\end{align*}
	which is an increasing function for $h\ge 0$. Using the quadratic formula, we can compute its inverse to be
	\begin{align}
	\begin{split}
		\varphi^{-1}(x) &= x + \frac{m}{4} - \sqrt{\left(x + \frac{m}{4}\right)^2-x^2} \\
		&>
		\frac{1}{2\left(x + m/4\right)}x^2
	\label{phi_inverse_estimate}
	\end{split}
	\end{align}
	where the strict inequality holds for any $x>0$, by Taylor expansion. Setting $x = 1/2C$ in \eqref{phi_inverse_estimate}, we get the bound
	\begin{align*}
		\varphi^{-1}\left(\frac{1}{2C}\right)
		>
		\frac{1}{4C + 2C^2 m},
	\end{align*}
	and therefore, by strict monotonicity of $\varphi$ we get
	\begin{align*}
		K(m,C):=  2C \varphi\left( \frac{1}{4C + 2C^2 m} \right) < 1.
	\end{align*}
	By monotonicity of $\varphi$, this also shows that whenever
	\begin{align*}
		H(f|M)\le  \frac{1}{4C + 2C^2 m},
	\end{align*}
	we have the bound $2C\varphi(H(f|M))\le K(m,C)$, which proves \eqref{bound_by_K(m,C)}, which proves the claim.

	\end{proof}

	We can now combine the estimates from Lemmas \ref{ContEstLemv2} and \ref{small_entropy_lemma} to get the bound
	\begin{align}
	\begin{split}
	\label{small_entropy_SqQ+_control}
		&\norm{S_qQ^+(f,f)}_X \\
		&\le
		\min\left(
			8\pi\norm{\phi}_{L^1}\norm{f}_X^2,\,
			\alpha\left(1+\frac1q\right) 
			\max\left(
				\frac{m}{\epsilon},\ 
				\frac{\norm{f}_{X}}{\log^+(m^{-1}\norm{f}_{X})+\epsilon}
			\right)
		\right)
	\end{split}
	\end{align}
	for some $\alpha<1$ and $\epsilon>0$, whenever 
	\begin{align*}
		H(f|M) \le \frac{1}{4C + 2C^2 m}
	\end{align*}
	holds. If $f\in L^\infty_{\mrm{loc}}([0,T_*),X)$ is a solution to \eqref{1D_Boltzmann} from Lemma \ref{local_lemma}, and if $f_{\mrm{in}}\in L^1(1+|v|^2)\cap L\log L(\TT_x^1\times\RR_v^3)$, then by Proposition \ref{regularity_proposition}, we have that $H(f(t)|M)$ is non-increasing in time and therefore \eqref{small_entropy_SqQ+_control} holds for all time $t\in [0,T_*)$. By applying this bound to the Duhamel formulation \eqref{Duhamel_Boltzmann} we get the estimate
	\begin{align}
		&\norm{f(t)}_X \nonumber
		\\
		&= \sup_{q\ge 0}\esssup_{x\in\TT_x^1} \int_{\RR_v^3}\! S_q f(t,x,v)\dd v \nonumber
		\\
		&\le \norm{S_t f_{\mrm{in}}}_X
		+
		\sup_{q\ge 0}\esssup_{x\in\TT_x^1} \int_{\RR_v^3}\!
			S_q \int_0^t\! S_{t-s}[Q^+(f,f)(s) - Q^-(f,f)(s)]\dd s \dd v  \nonumber
			\\
		&\le
		\norm{f_{\mrm{in}}}_X
		+
		\int_0^t\! \norm{S_{t-s}Q^+(f,f)(s)}_X\dd s \nonumber
		\\
		\begin{split}
		&\le
		\norm{f_{\mrm{in}}}_X 
		\\
		&\quad+
		\int_0^t\!
		\min\left(
			c_1\norm{f}_X^2,\,
			\alpha\left(1+\frac{1}{t-s}\right) 
			\max\left(
				\frac{m}{\epsilon},\ 
				\frac{\norm{f}_{X}}{\log^+(m^{-1}\norm{f}_{X})+\epsilon}
			\right)
		\right)
		\dd s
		\label{norm(f)_X_small_entropy_bound}
		\end{split}
	\end{align}	
	for $c_1 = c(B)$. In the first equality we have used non-negativity $f\ge 0$  to discard the loss term $Q^-$, and then we have used that $\norm{S_tf_{\mrm{in}}}_X$ is non-increasing in  $t$. We can now close this estimate, using a Gronwall-type result which was stated without proof in \cite{BiryukCraigPanferov}.

		\begin{lemma}[Adapted from \cite{BiryukCraigPanferov}]\label{small_entropy_ODE}
			 Let $\varphi\in L^\infty_{\mrm{loc}}([0,T_*))$ satisfy the bound $\varphi\ge m$ and the integral inequality
		\begin{align}\label{entropy_int_ineq}
			\varphi(t)\le
			c_0
			+
			\int_0^t\!
			\min\left\{
				c_1 \varphi(s)^2,
				\left( 
					\frac{\alpha}{t-s} + c_2
				\right)
				\max\left(
					\frac{m}{\epsilon},\ 
					\frac{\varphi(s)}{\log(m^{-1}\varphi(s))+\epsilon}
				\right)
			\right\}
			\dd s 
		\end{align}
		for $c_0,c_1,c_2,c_3,\alpha,\epsilon>0$ constants. Suppose $\alpha<1$. Then we have the bound
		\begin{align*}
			\varphi(t)\le e^{C\sqrt{1+t}}
		\end{align*}
		for $C = C(c_0,c_1,c_2,\epsilon,m,\alpha)$, and this constant blows up as $\alpha\to 1$.
			
		\end{lemma}
	
		\begin{proof}
			
			Our proof proceeds by splitting the integral in \eqref{entropy_int_ineq} at $t-\delta$ for some $\delta$, and by bounding the two terms of the minimum separately. We first define 
			\begin{align*}
				\varphi_*(t) = \max\Big( \sup_{s\in [0,t]}\varphi(s),\: K\, (1+t)^p \Big),
			\end{align*}
			where $K,p>0$ will be chosen later, and $K$ will be taken sufficiently large. We  then define
			\begin{align*}
				\delta = m/\varphi_*(t)\log\varphi_*(t).
			\end{align*}
		We can use \eqref{entropy_int_ineq} to bound
		\begin{align*}
			\varphi(t)
			&\le
			c_0
			+
			\int_{t-\delta}^t c_1\varphi(s)^2 \dd s
			+
			\int_0^{t-\delta}\!
			\left( 
			 \frac{\alpha}{t-s} + c_2
			\right)
			\max\left(\frac{m}{\epsilon},\ 
				\frac{\varphi(s)}{\log(m^{-1}\varphi(s))+\epsilon}
			\right)
			\dd s
			\\
			&\le
			c_0+\delta c_1\varphi_*(t)^2
			+
			\int_0^{t-\delta}\!
				\frac{\alpha}{t-s}
				\max\left(\frac{m}{\epsilon},\ 
					\frac{\varphi(s)}{\log(m^{-1}\varphi(s))+\epsilon}
				\right)
				\dd s
			\\
			&\qquad\qquad\qquad\qquad\qquad\qquad+
			c_2 \int_0^t\! 
				 \max\left(\frac{m}{\epsilon},\ 
					 \frac{\varphi(s)}{\log(m^{-1}\varphi(s))+\epsilon}
				 \right)
				\dd s
			\\
			&\le
			c_0 + \frac{c_1 m}{\log\varphi_*(t)}\varphi_*(t) \\
			&\qquad\qquad\qquad
			+ \alpha\log(t\delta^{-1})\sup_{[0,t]}
				\max\left(\frac{m}{\epsilon},\ 
					\frac{\varphi(s)}{\log(m^{-1}\varphi(s))+\epsilon}
				\right)
			\\
			&\qquad\qquad\qquad\qquad\qquad\qquad
			+
			c_2\int_0^t\! 
			\max\left(\frac{m}{\epsilon},\ 
				\frac{\varphi(s)}{\log(m^{-1}\varphi(s))+\epsilon}
			\right)
			\dd s
			\\
			&\le
			c_0 + \frac{c_1 m}{\log\varphi_*(t)}\varphi_*(t) \\
			&\qquad\qquad\qquad
			+
			\alpha\log(t m^{-1}\varphi_*(t)\log\varphi_*(t))
				\max\left(
					\frac{m}{\epsilon},\ 
					\frac{\varphi_*(t)}{\log(m^{-1}\varphi_*(t))+\epsilon} 
				\right)
			\\
			&\qquad\qquad\qquad\qquad\qquad
			+
			c_2\int_0^t\! 
			 \max\left(\frac{m}{\epsilon},\ 
				 \frac{\varphi(s)}{\log(m^{-1}\varphi(s))+\epsilon}
			 \right)
			\dd s.
		\end{align*}
	
	In the last inequality, we have used the fact that the function
	\begin{align*}
		[m,\infty)\mapsto\RR_{\ge0},
		\quad
		\varphi
		\mapsto
		\max\left(\frac{m}{\epsilon},\ 
		\frac{\varphi}{\log(m^{-1}\varphi)+\epsilon}
		\right)
	\end{align*}
	is increasing in $\varphi$, and therefore commutes with suprema. For $K = K(m,\epsilon)$ sufficiently large, we can use $\varphi_*\ge K$ to write
	\begin{align}
	\begin{split}
		\label{monotonic_phi_logphi_bound}
		\sup_{s\in[0,t]}
		\max\left(\frac{m}{\epsilon},\ 
		\frac{\varphi(s)}{\log(m^{-1}\varphi(s))+\epsilon}
		\right)
		&\le
		\max\left(\frac{m}{\epsilon},\ 
			\frac{\varphi_*(t)}{\log(m^{-1}\varphi_*(t))+\epsilon}
		\right)
		\\
		&=
		\frac{\varphi_*(t)}{\log(m^{-1}\varphi_*(t))+\epsilon},
	\end{split}
	\end{align}
	and we can then bound the term
	\begin{align*}
		&
		\alpha\log(t m^{-1}\varphi_*(t)\log\varphi_*(t))
		\frac{\varphi_*(t)}{\log(m^{-1}\varphi_*(t))+\epsilon}
		\\
		&\quad\le
		\alpha
			\left[
				\log t + \log(m^{-1}\varphi_*(t)) + \log\log\varphi_*(t)
			\right]
		\frac{\varphi_*(t)}{\log(m^{-1}\varphi_*(t))+\epsilon} 
		\\
		&\quad\le
		\frac{\alpha\log t}{\log(m^{-1}\varphi_*(t))}\varphi_*(t) + \alpha\varphi_*(t)
		+
		\alpha\frac{\log\log\varphi_*(t)}{\log(m^{-1}\varphi_*(t))}\varphi_*(t).
	\end{align*}
	If we choose $p=2\alpha/(1-\alpha)$, then using that $\varphi_*(t)\ge K(1+t)^p$, we have that
	\begin{align*}
		\frac{\alpha\log t}{\log(m^{-1}\varphi_*(t))}
		&\le \frac{\alpha\log t}{\log(m^{-1}K) + p\log(1+t)} \\
		&\le \frac{\alpha}{p} = \frac{1-\alpha}{2},
	\end{align*}
	where we have assumed $K\ge m$. By choosing $K$ sufficiently large, we can also ensure that
	\begin{align*}
			\frac{c_1 m}{\log\varphi_*(t)} + 
			\alpha\frac{\log\log\varphi_*(t)}{\log(m^{-1}\varphi_*(t))}\le \frac{1-\alpha}{4},
	\end{align*}
	so that by combining terms, we can estimate
	\begin{align}
		\varphi(t)
		&\le
		c_0 + \frac{3+\alpha}{4}\varphi_*(t) + c_2 \int_0^t\! 
		 \max\left(\frac{m}{\epsilon},\ 
			 \frac{\varphi(s)}{\log(m^{-1}\varphi(s))+\epsilon}
		 \right)
		\dd s \nonumber
		\\
		&\le
		c_0  + \frac{3+\alpha}{4}\varphi_*(t) + c_2 \int_0^t\! 
		\frac{\varphi_*(s)}{\log(m^{-1}\varphi_*(s))}
		\dd s
		\label{varphistar_second_bound}
	\end{align}
	where in the second inequality we have used the bound \eqref{monotonic_phi_logphi_bound}. Since the terms in \eqref{varphistar_second_bound} are non-decreasing in $t$, we can take suprema in $t$ of both sides of the inequality to get
	\begin{align}
		\label{control_supphi_by_phistar}
		\sup_{s\in[0,t]}\varphi(s)
		\le
			c_0  + \frac{3+\alpha}{4}\varphi_*(t) + c_2 \int_0^t\! 
		\frac{\varphi_*(s)}{\log(m^{-1}\varphi_*(s))}
		\dd s.
	\end{align}
	In order to replace the left hand side with $\varphi_*(t)$, we have to prove the bound
	\begin{align}
		\label{polynomial_mollification_of_phi}
		K(1+t)^p \le C_0 + \frac{3+\alpha}{4}\varphi_*(t) 
		+
		c_2 \int_0^t\! 
		\frac{\varphi_*(s)}{\log(m^{-1}\varphi_*(s))}
		\dd s
	\end{align}
	for $C_0$ sufficiently large. We bound
	\begin{align*}
		C_0
		+
		c_2 \int_0^t\! 
		\frac{\varphi_*(s)}{\log(m^{-1}\varphi_*(s))}
		\dd s
		&\ge 
		C_0
		+
		c_2\int_0^t\!\frac{ K(1+s)^p}{\log(K/m) + p\log(1+s)} \\
		&\ge
		C_0 + C' \int_0^t (1+s)^{p -1/2}\dd s \\
		&=
		C_0 + C'' (1+t)^{p+1/2}
	\end{align*}
	where we have used the bound
	\begin{align*}
		 \log(K/m) + p\log(1+s)\le C(1+\log(1+s))\le C'\sqrt{1+s}
	\end{align*}
	for $C'$ depending on $K$. But by taking $C_0$ sufficiently large, we have that
	\begin{align*}
		C_0 + C'' (1+t)^{p + 1/2} \ge K (1+t)^p
	\end{align*}
	for all $t\ge 0$, which proves \eqref{polynomial_mollification_of_phi}. By combining this with \eqref{control_supphi_by_phistar} we get
	\begin{align*}
		\varphi_*(t) \le C_0 + \frac{3+\alpha}{4}\varphi_*(t) + c_2\int_0^t\! 
		\frac{\varphi_*(s)}{\log(m^{-1}\varphi_*(s))}
		\dd s,
	\end{align*}
	and therefore that
	\begin{align*}
		\varphi_*(t) \le C_0' + C_2'\int_0^t\! 
		\frac{\varphi_*(s)}{\log(m^{-1}\varphi_*(s))}
		\dd s.
	\end{align*}
	This can then be solved by standard ODE estimates, which give that
	\begin{align*}
		\varphi_*(t) \le e^{C\sqrt{1+t}},
	\end{align*}
	where $C = C(c_0,c_1,c_2,\epsilon,m,\alpha)$, and since $\varphi(t)\le\varphi_*(t)$ by construction, this concludes the proof.

		\end{proof}

	By applying Lemma \ref{small_entropy_ODE} to the bound \eqref{norm(f)_X_small_entropy_bound}, we see that under the conditions of Theorem \ref{Thm1}, for any solution $f\in L^\infty_{\mrm{loc}}([0,T_*),X)$ to \eqref{1D_Boltzmann} from Lemma \ref{local_lemma}, we have the bound
	\begin{align*}
		\norm{f(t)}_X\le e^{C\sqrt{1+t}}
	\end{align*}
	for $C = C(\norm{\phi}_{L^1},\norm{(1-r^{-1})B}_{L^\infty_{r,\theta}},\norm{f_{\mrm{in}}}_X, m)$. Therefore for any finite time $T\le T_*$ we have a uniform bound $f \in L^\infty([0,T),X)$, so by the blowup criterion of Lemma \ref{local_lemma} the solution $f$ must exist for all time, which concludes the proof of Theorem \ref{Thm1}.

	\section{Large Data} \label{Thm2_Proof_Section}

	We now present a proof of Theorem \ref{Thm2}. Our strategy is similar to the proof of Theorem \ref{Thm1}, but we now require an analogue of the bound \eqref{norm(f)_X_small_entropy_bound} which will hold for large data. In this section, we impose hypotheses \textbf{(H1)} and \textbf{(H2)} on the collision kernel $B$. This will allow us to exploit the Bony functional, which we define in \eqref{Bony_functional_definition}. This functional was used by Cercignani to construct weak solutions to \eqref{1D_Boltzmann} in \cite{Cercignani05}, using the following key estimate.
	
	\begin{lemma}[\cite{Cercignani05}] \label{Cercignani05_Lemma2.4}
		Let $f\in C^\infty([0,T]\times\RR_x^1\times\RR_v^3)\cap C([0,T], L^1(1+|v|^2))$ be a classical solution to \eqref{1D_Boltzmann} on the real line, and let the collision kernel $B$ satisfy hypotheses \textbf{(H1)} and \textbf{(H2)}. Then we have the estimate
		\begin{align*}
			&
			\int_0^T\!\int\limits_{\RR^1_x\times\RR^3_v\times\RR^3_{v_*}\times\SSS^2_\sigma}\!\!\!
			B\left( |v-v_*|,\sigma\cdot\frac{v-v_*}{|v-v_*|}\right)
			|v-v_*|^2 f(x,v,t) f(x,v_*,t)\dd\sigma\dd v_*\dd v\dd x\dd t
			\\
			&\qquad\qquad\qquad
			\le 
			K_0<\infty
		\end{align*}
		where $K_0 = K_0(f_{\mrm{in}},B)$ is a constant independent of time.
	\end{lemma}
	
	This estimate was then used by Biryuk, Craig, and Panferov for their large data global well-posedness result \cite{BiryukCraigPanferov}. In this section, we show how the work \cite{BiryukCraigPanferov} can be extended to obtain quantitative estimates.

	Our first challenge is to extend Lemma \ref{Cercignani05_Lemma2.4} to solutions $f\in L^\infty_{\mrm{loc}}([0,T_*),X)\cap C([0,T_*),L^1(1+|v|^2))$. Since the space $X$ contains functions $g$ with discontinuous densities $\int g\dd v$, which cannot be approximated in the $X$ norm by smooth functions, we are not able to use continuous dependence on the initial data from the local well-posedness theory. One might hope to instead approximate $f$ weakly by smooth solutions $f^\epsilon$ and prove Lemma \ref{Cercignani05_Lemma2.4} for $f$ using a convexity argument. However, since we are not assuming that $f\log f\in L^1(\Omega_x^1\times\RR_v^3)$, we cannot control the entropy dissipation, and therefore the weak-strong uniqueness result of \cite{Lions94} is not available to us. Therefore we cannot guarantee that $f^\epsilon_{\mrm{in}}\rightharpoonup f_{\mrm{in}}$ will imply that $f^\epsilon\rightharpoonup f$ holds on $[0,T]\times\Omega_x^1\times\RR_v^3$ up to subsequences, so we will instead prove Lemma \ref{Cercignani05_Lemma2.4} directly in low regularity. Our proof will closely mirror that of \cite{Cercignani05}, with modifications to directly treat solutions $f$ in $ L^\infty_{\mrm{loc}}([0,T_*),X) \cap C([0,T_*),L^1(1+|v|^2))$. To do this, we will rely on the doubling of variables method, developed by Kruzkov \cite{Kruzkov70} to control weak solutions to scalar conservation laws.
	
	Our next challenge is to use Lemma \ref{Cercignani05_Lemma2.4} to obtain quantitative growth estimates for $f$ in $X$. In \cite{BiryukCraigPanferov} it was shown that Lemma \ref{Cercignani05_Lemma2.4} implies $f\in L^\infty_{\mrm{loc}}([0,\infty),X)$. This was proved using an integral inequality estimate for $\norm{f(t)}_X$ with coefficients depending on the integral in Lemma \ref{Cercignani05_Lemma2.4}, and by closing this estimate without a rate. In this paper we prove sharp growth estimates for the integral inequality in \cite{BiryukCraigPanferov}, which allow us to show that $\norm{f(t)}_X\le C e^{C t}$ when $\Omega_x^1=\TT_x^1$, and $\norm{f(t)}_X\to 0$ as $t\to\infty$ when $\Omega_x^1 = \RR_x^1$.
	
	We start by adapting Lemma \ref{Cercignani05_Lemma2.4} as follows.
	
	\begin{lemma} 
		\label{integral_of_A_lemma}
		Let $f\in L^\infty_{\mrm{loc}}([0,T_*),X)\cap C([0,T_*),L^1(1+|v|^2))$ be a solution to \eqref{1D_Boltzmann} on $\Omega_x^1\times\RR_v^3$, defined up to maximal time $0<T_*\le \infty$. Suppose the collision kernel $B$ satisfies the hypotheses \textbf{(H1)} and \textbf{(H2)}. Writing
		\begin{align*}
			A(t) = 
			\int\limits_{\RR^1_x\times\RR^3_v\times\RR^3_{v_*}\times\SSS^2_\sigma}\!\!\!\!
			B\left( |v-v_*|,\sigma\cdot\frac{v-v_*}{|v-v_*|}\right)
			|v-v_*|^2 f(x,v,t) f(x,v_*,t)\dd\sigma\dd v_*\dd v\dd x,
		\end{align*}
		we have the estimate
		\begin{align*}
			\int_0^T A(t)\dd t
			\le
			\begin{cases}
				K_0 (1+T) & \Omega_x^1 = \TT_x^1
				\\
				K_0 & \Omega_x^1 = \RR_x^1
			\end{cases},
		\end{align*}
		where $K_0 = K_0\big(\norm{f_\mrm{in}}_{L^1(1+|v|^2)},B\big)$.
	\end{lemma}
	
	Before we prove Lemma \ref{integral_of_A_lemma}, we introduce some functionals which will be useful for our calculations. We define the one-dimensional Green's functions $\ggg:\Omega_x^1\to\RR$ as
	\begin{align*}
		\ggg(x) = -\frac{|x|}{2}
	\end{align*}
	when $\Omega_x^1 = \RR_x^1$, and as
	\begin{align*}
		\ggg(x) = -\frac{|x|}{2} + \frac{x^2}{2},\qquad x\in\left[-\frac12,\frac12\right]
	\end{align*}
	when $\Omega_x^1=\mathbb{T}_x^1$, where we extend $\ggg$ to $\TT_x^1$ by periodicity. Differentiating gives the distributional identities
	\begin{align*}
		\ggg'(x) = 
		\begin{cases}
			-\frac12 \operatorname{sgn}(x) & \Omega_x^1 = \RR_x^1\\
			-\frac12\operatorname{sgn}(x) + x 
			& \Omega_x^1 = \TT_x^1,\quad x\in[-1/2,1/2]
		\end{cases}
	\end{align*}
	and
	\begin{align}\label{ggg''_formula}
		\ggg''(x) =
		\begin{cases}
			-\delta_0(x) & \Omega_x^1 = \RR_x^1\\
			-\delta_0(x) + 1
			& \Omega_x^1 = \TT_x^1
		\end{cases},
	\end{align}
	where $\delta_0$ is a Dirac delta function at $0$.
	
	We can now define the \emph{Bony functional} as
	\begin{align}\label{Bony_functional_definition}
		L(t)
		=
		\int\limits_{\Omega_x^1\times\Omega_{x_*}^1\times\RR_v^3\times\RR_{v_*}^3}\!\!\!\!\!\!
		f(t,x,v)f(t,x_*,v_*)\ggg'(x-x_*)(v_1-v_{*,1})\dd v_*\dd v\dd x_*\dd x
	\end{align}
	We also define the \emph{Bony dissipation}
	\begin{align*}
		D_B(t) = 
		\int_{\Omega_x^1\times\RR_v^3\times\RR_{v_*}^3}\!\!\!\!
		f(t,x,v)f(t,x,v_*)(v_1-v_{*,1})^2\dd v_*\dd v\dd x \ge 0
	\end{align*}
	and the periodic error
	\begin{align*}
		\ell(t) = \int_{\TT_x^1\times\TT_{x_*}^1\times\RR_v^3\times\RR_{v_*}^3}\!\!\!\!
		f(t,x,v)f(t,x_*,v_*)(v_1-v_{*,1})^2\dd v_*\dd v\dd x_*\dd x \ge 0.
	\end{align*}
	We can then prove the following dissipation inequality.
	
	\begin{lemma} 
		\label{Bony_dissipation_lemma}
		For $f$ as given in Lemma \ref{integral_of_A_lemma}, we have the inequalities
		\begin{align}
			\label{Bony_inequality}
			L(t) + \int_0^t \! D_B(s)\dd s 
			\le
			L(0) + 
			\begin{cases}
				0& 													\Omega_x^1 = \RR_x^1 \\
				\int_0^t\!\ell(s)\dd s &		\Omega_x^1 = \TT_x^1
			\end{cases}.
		\end{align}
	\end{lemma}
	
	We note that Lemma \ref{Bony_dissipation_lemma} was proved for smooth data in \cite{Cercignani05}. We now present a proof that avoids regularity issues associated with $L^\infty_x$ class data.
	
	\begin{proof}
			To prove this directly for non-smooth $f$, we use the doubling of variables trick developed by Kruzkov \cite{Kruzkov70}. We note that $f\in L^\infty_{\mrm{loc}}([0,T_*),X)\cap C([0,T_*),L^1(1+|v|^2))$ is a weak solution to Boltzmann in the following sense. If $\varphi(t,x,v)$ is a \emph{suitable test function}, which we take to mean that it satisfies the bounds
		\begin{align*}
			(1+|v|^2)^{-1}\varphi \in L^\infty_{\mrm{loc}}([0,T_*),L^\infty),
			\qquad
			(1+|v|^2)^{-1}(\p_t + v_1\p_x)\varphi
			\in L^\infty_{\mrm{loc}}([0,T_*),L^\infty),
		\end{align*}
		then the weak formulation
		\begin{align}\label{weak_formulation}
			\begin{split}
				-\left.\int_{\Omega_x^1\times\RR_v^3}\!
				f\varphi\dd v\dd x\right|_{t=0}^T
				+
				\int_0^T\! \int_{\Omega_x^1\times\RR_v^3}\!
				&f(\p_t\varphi+ v_1\p_x\varphi)\dd v\dd x\dd t \\
				&\qquad
				+ \int_0^T\! \int_{\Omega_x^1\times\RR_v^3}\!
				Q(f,f)\varphi\dd v\dd x\dd t 
				=0
			\end{split}
		\end{align}
		holds for any $0<T<T_*$. To see this, we integrate the mild formulation $\p_t f^\# = Q(f,f)^\#$ against $\varphi^\#$ and integrate by parts in time to get
		\begin{align*}
			-\left.\int_{\Omega_x^1\times\RR_v^3}\! f^\#\varphi^\#\dd v\dd x\right|_{t=0}^T
			+
			\int_0^T\! \int_{\Omega_x^1\times\RR_v^3}\!
			&f^\#\p_t\varphi^\#\dd v\dd x\dd t 
			\\
			&\qquad
			+ \int_0^T\! \int_{\Omega_x^1\times\RR_v^3}\!
			Q(f,f)^\#\varphi^\#\dd v\dd x\dd t 
			=0
		\end{align*}
		where we use $Q(f,f)\in L^\infty([0,T],L^1(1+|v|^2))$ from Lemma \ref{XtoL^1_ell_Lemma} and the identity $\p_t(\varphi^\#) = (\p_t\varphi + v_1\p_x\varphi)^\#$ to derive \eqref{weak_formulation}. We regularize $\ggg_\epsilon = \ggg * \eta_\epsilon$ where $\eta_\epsilon\in C^\infty$ is an approximate identity with support in $[-\epsilon,\epsilon]$. We use this to define the regularized Bony functional 
		\begin{align*}
			L^\epsilon(t)
			=
			\int_{\Omega_x^1\times\Omega_{x_*}^1\times\RR_v^3\times\RR_{v_*}^3}\!
			f(t,x,v)f(t,x_*,v_*)\ggg_\epsilon'(x-x_*)(v_1-v_{*,1})\dd v_*\dd v\dd x_*\dd x.
		\end{align*}
		We now define
		\begin{align*}
			\psi(t,x,v) = 
			\int_{\Omega_{x_*}^1\!\times\RR_{v_*}^3}f(t,x_*,v_*) \ggg_\epsilon'(x-x_*)(v_1 - v_{*,1})\dd v_*\dd x_*
			= \psi_0(t,x) + \psi_1(t,x)v_1.
		\end{align*}
		We note that $(x_*,v_*)\mapsto g_\epsilon'(x-x_*)(v_1 - v_{*,1})$ are suitable test functions for all $(x,v)$. We can then apply \eqref{weak_formulation} to get the identity
		\begin{align*}
			\p_t\psi(t,x,v) 
			&= \int_{\Omega_{x_*}^1\!\times\RR_{v_*}^3} f(t,x_*,v_*) v_{*,1}(-\ggg_\epsilon)''(x-x_*)(v_1 - v_{*,1})\dd v_*\dd x_* \\
			&\qquad
			+
			\int_{\Omega_{x_*}^1\!\times\RR_{v_*}^3} Q(f,f)(t,x_*,v_*)
			\ggg_\epsilon'(x-x_*)(v_1 - v_{*,1})\dd v_*\dd x_* \\
			&= \int_{\Omega_{x_*}^1\!\times\RR_{v_*}^3} f(t,x_*,v_*) (-v_{*,1})\ggg_\epsilon''(x-x_*)(v_1 - v_{*,1})\dd v_*\dd x_*,
		\end{align*}
		where the integral involving $Q(f,f)$ vanishes by conservation of mass and momentum. We can also calculate
		\begin{align*}
			v_1\p_x\psi(t,x,v)
			&=
			\int f(t,x_*,v_*) v_1\ggg_\epsilon''(x-x_*)(v_1 - v_{*,1})\dd v_*\dd x_*.
		\end{align*}
		Using that $f\in L^\infty_{\mrm{loc}}([0,T_*), L^1(1+|v|^2))$ and $\ggg_\epsilon',\ggg_\epsilon''\in L^\infty$, we see that $\psi$ is a suitable test function, so we can again apply \eqref{weak_formulation} to get
		\begin{align}
			\label{Bony_difference_term}
			&\left.\int_{\Omega_{x}^1\times\RR_{v}^3\times\Omega_{x_*}^1\!\times\RR_{v_*}^3}\!\!
			f(t,x,v)f(t,x_*,v_*)\ggg_\epsilon'(x-x_*)(v_1 - v_{*,1})\dd v_*\dd x_*\dd v\dd x \right|_{t=0}^T \\
			& = 
			\left. \int_{\Omega_{x}^1\times\RR_{v}^3}\!
			f(t,x,v)\psi(t,x,v)\dd v\dd x \right|_{t=0}^T \nonumber
			\\
			&=
			\int_0^T\!\int_{\Omega_{x}^1\times\RR_{v}^3}\!
			f(t,x,v)(\p_t + v_1\p_x)\psi(t,x,v)\dd v\dd x\dd t  \nonumber
			\\
			&\quad\qquad 
			+ \int_0^T\!\int_{\Omega_{x}^1\times\RR_{v}^3}\!
			Q(f,f)(t,x,v)\psi(t,x,v)\dd v\dd x\dd t \nonumber
			\\
			\begin{split} 
				&=
				\int_0^T\!\int_{\Omega_{x}^1\times\RR_{v}^3\times\Omega_{x_*}^1\!\times\RR_{v_*}^3}\!\!\!
				f(t,x,v)f(t,x_*,v_*)
				\\
				&\qquad\qquad\qquad\qquad\qquad\qquad\qquad
				\cdot \ggg_\epsilon''(x-x_*)(v_1-v_{*,1})^2
				\dd v_*\dd x_*\dd v\dd x \dd t,
				\label{Bony_derivative_term}
			\end{split}
		\end{align}
		where we have used that $\psi(t,x,v)$ is an affine function of $v$ to cancel the collision term $Q(f,f)$. By convolution, we have the identities
		\begin{align}\label{ggg_epsilon_=_eta_epsilon}
			\ggg_\epsilon'' =
			\begin{cases}
				-\eta_\epsilon & \Omega_x^1 = \RR_x^1\\
				-\eta_\epsilon + 1
				& \Omega_x^1 = \TT_x^1
			\end{cases}.
		\end{align}
		We can then rewrite and rearrange the terms in \eqref{Bony_difference_term} and \eqref{Bony_derivative_term} to get
		\begin{align}
			&L^\epsilon(T) - L^\epsilon(0) 
			\nonumber
			\\
			&+
			\int_0^T\!\int\limits_{\Omega_{x}^1\times\RR_{v}^3\times\Omega_{x_*}^1\!\times\RR_{v_*}^3}\!\!\!
			\!\!\!
			f(t,x,v)f(t,x_*,v_*)\eta_\epsilon(x-x_*)(v_1-v_{*,1})^2\dd v_* \dd x_*\dd v\dd x\dd t
			\label{DBepsilon}
			\\
			&\qquad\qquad\qquad\qquad\qquad\qquad\qquad\qquad
			=
			\begin{cases}
				0 & \Omega_x^1 = \RR_x^1  \\
				\int_0^T \ell(t)\dd t & \Omega_x^1 = \TT_x^1
			\end{cases}.
			\nonumber
		\end{align}
		
		We see by dominated convergence that $L^\epsilon(t)\to L(t)$ for all $t\in [0,T_*)$, where we use that $f\in C([0,T_*),L^1(1+|v|^2))$ and $\ggg_\epsilon'\in L^\infty$ uniformly in $\epsilon$. We then see that \eqref{DBepsilon} is uniformly bounded in $\epsilon$. By rewriting \eqref{DBepsilon} as
		\begin{align*}
			\int_0^T\int_{\RR_{v}^3\times\RR_{v_*}^3}\!\!
			\left(
			\int_{\Omega_x^1\times\Omega_{x_*}^1}\!\! f(t,x,v)f(t,x_*,v_*)\eta_\epsilon(x-x_*)\dd x_*\dd x
			\right)
			(v_1-v_{*,1})^2
			\dd v_*\dd v\dd x\dd t
		\end{align*}
		and taking limits as $\epsilon\to0$, we see that the integral in $x_*,x$ converges pointwise for almost every $v_*,v,t$, which by Fatou's lemma gives the inequality \eqref{Bony_inequality}.
	\end{proof}
	
	The rest of the proof of Lemma \ref{integral_of_A_lemma} follows arguments in \cite{Cercignani05}, which we include here for completeness.
	
	\begin{proof}[Proof of Lemma \ref{integral_of_A_lemma}]
		We first note that we can bound 
		\begin{align*}
				|L(t)|+|\ell(t)|
				&\le C\norm{f(t)}_{L^1}\norm{f(t)}_{L^1(1+|v|^2)}
				\\
				&
				\le C \norm{f_{\mrm{in}}}_{L^1}\norm{f_{\mrm{in}}}_{L^1(1+|v|^2)},
		\end{align*}
		which gives the bound
		\begin{align}
			\int_0^T\! D_B(t)\dd t\le
			\begin{cases}
				C\, (1+T) & \Omega_x^1 = \TT_x^1
				\\
				C & \Omega_x^1 = \RR_x^1
			\end{cases}
		\end{align}
		for $C = C(\norm{f_{\mrm{in}}}_{L^1(1+|v|^2)})$. If we define the average velocity
		\begin{align}
			\label{average_velocity}
			u(x,t) = \left(\int_{\RR^3_v}\!f(t,x,v)\dd v\right)^{-1}\int_{\RR_v^3}\! v f(t,x,v)\dd v,
		\end{align}
		then we can rewrite the Bony dissipation as
		\begin{align*}
			D_B(t)
			&=
			\int_{\Omega_x^1\times\RR_v^3\times\RR_{v_*}^3}\!\!\!\!
			f(t,x,v)f(t,x,v_*)(v_1^2 - 2 v_1 v_{*,1}+v_{*,1}^2)\dd v_*\dd v\dd x
			\\
			&= 2 \int_{\Omega_x^1\times\RR_v^3\times\RR_{v_*}^3}\!\!\!\!
			f(t,x,v)f(t,x,v_*)(v_1^2 -  v_1 v_{*,1})\dd v_*\dd v\dd x
			\\
			&=
			2 \int_{\Omega_x^1\times\RR_v^3\times\RR_{v_*}^3}\!\!\!\!
			f(t,x,v)f(t,x,v_*)(v_1^2 -  u_1(x,t)^2)\dd v_*\dd v\dd x
			\\
			&=
			2 \int_{\Omega_x^1\times\RR_v^3\times\RR_{v_*}^3}\!\!\!\!
			f(t,x,v)f(t,x,v_*)(v_1^2 - 2 v_1 u_1(x,t)+ u_1(x,t)^2)\dd v_*\dd v\dd x
			\\
			&=
			2\int_{\Omega_x^1\times\RR_v^3\times\RR_{v_*}^3}\!\!\!\!
			f(t,x,v)f(t,x,v_*)(v_1-u_1(x,t))^2\dd v_*\dd v\dd x
		\end{align*}

		where the second equality follows by symmetry in $v$ and $v_*$, and where we have twice used the identity \eqref{average_velocity}.
		
		We now multiply the equation \eqref{1D_Boltzmann} by $v_1^2$ and integrate the resulting equation
		\begin{align*}
			v_1^2\p_t f + v_1^3\p_x f = v_1^2 Q(f,f)
		\end{align*}
		over $[0,T]\times\Omega_x^1\times\RR_v^3$ for $T<T_*$, to get 
		\begin{align*}
			\left. \int_{\Omega_x^1\times\RR_v^3}\! f(t,x,v) v_1^2\dd v\dd x\right|_0^T
			&=
			\int_0^T\!\!\int_{\Omega_x^1\times\RR_v^3}\!  v_1^2 Q(f,f)\dd v\dd x\dd t,
		\end{align*}
		which gives us the bound
		\begin{align*}
			\int_0^T\!\!\int_{\Omega_x^1\times\RR_v^3}\!  (v_1-u_1(x,t))^2 Q(f,f)\dd v\dd x\dd t
			&=
			\int_0^T\!\!\int_{\Omega_x^1\times\RR_v^3}\!  v_1^2 Q(f,f)\dd v\dd x\dd t
			\\
			&
			\le
			\int_{\Omega_x^1\times\RR_v^3}\! f(T,x,v) v_1^2\dd v\dd x
			\\
			&\le
			\int_{\Omega_x^1\times\RR_v^3}\! f(T,x,v) |v|^2\dd v\dd x
			\\
			&= \int_{\Omega_x^1\times\RR_v^3}\! f_{\mrm{in}}(x,v) |v|^2\dd v\dd x
			\\
			&
			\le
			\norm{f_{\mrm{in}}}_{L^1(1+|v|^2)}\le C
		\end{align*}
		where we have used the cancellation identities \eqref{collision_kernel_cancellation} in the first equality. We can then split $Q = Q^+ - Q^-$ and bound
		\begin{align}
			&\int_0^T\!\!\int_{\Omega_x^1\times\RR_v^3}\!  (v_1-u_1)^2 Q^+(f,f)\dd v\dd x\dd t
			\label{gain_term_a_priori_bound}
			\\
			&\qquad
			\le C + \int_0^T\!\!\int_{\Omega_x^1\times\RR_v^3}\!  (v_1-u_1)^2  Q^-(f,f)\dd v\dd x\dd t
			\nonumber
			\\
			&\qquad
			= C +
			\int_0^T\!\!\int\limits_{\Omega_x^1\times\RR_v^3\times\RR_{v_*}^3\times\SSS_\sigma^2}\!\!\!\! 
				(v_1-u_1)^2 B\ f(t,x,v)f(t,x,v_*)
			\dd\sigma\dd v_*\dd v\dd x\dd t
			\label{loss_term_a_priori_bound}
			\\
			&\qquad
			\le C + C\norm{B}_{L^\infty_r L^1_\theta} 
				\int_0^T\!\!\int\limits_{\Omega_x^1\times\RR_v^3\times\RR_{v_*}^3}\!\!\!\! 
			(v_1-u_1)^2 f(t,x,v)f(t,x,v_*)
			\dd v_*\dd v\dd x\dd t
			\nonumber
			\\
			&\qquad
			\le
			C' + C' \int_0^t\! D_B(t)\dd t
			\nonumber
		\end{align}
		for some constant $C' = C'(\norm{f_{\mrm{in}}}_{L^1(1+|v|^2)},B)$. But we can rewrite 
		\begin{align*}
			&\int_0^T\!\!\int_{\Omega_x^1\times\RR_v^3}\!  (v_1-u_1)^2 Q^+(f,f)\dd v\dd x\dd t
			\\
			&
			\qquad\qquad
			=
			\int_0^T\!\!\int_{\Omega_x^1\times\RR_v^3\times\RR_{v_*}^3\times\SSS_\sigma^2}\!\!\!
			(v_1-u_1)^2 B\,  f(t,x,v')f(t,x,v_*')\dd\sigma\dd v_*\dd v\dd x\dd t
			\\
			&
			\qquad\qquad
			=
			\int_0^T\!\!\int_{\Omega_x^1\times\RR_v^3\times\RR_{v_*}^3\times\SSS_\sigma^2}\!\!\!
			(v_1'-u_1)^2 B\,  f(t,x,v)f(t,x,v_*)\dd\sigma\dd v_*\dd v\dd x\dd t
			\\
			&\qquad\qquad
			=
			\int_0^T\!\!\int_{\Omega_x^1\times\RR_v^3\times\RR_{v_*}^3\times\SSS_\sigma^2}\!\!\!
			\left(
				\frac{v_1 - u_1}{2} + \frac{v_{*,1}-u_1}{2} + \sigma_1\frac{|v-v_*|}{2}
			\right)^2 
			\\
			&\qquad\qquad\qquad\qquad\qquad\qquad\qquad\qquad\qquad\cdot
			B\,  f(t,x,v)f(t,x,v_*)\dd\sigma\dd v_*\dd v\dd x\dd t
			\\
			&\qquad\qquad
			\ge
			\frac{1}{8}\int_0^T\!\!\int_{\Omega_x^1\times\RR_v^3\times\RR_{v_*}^3\times\SSS_\sigma^2}\!\!\!
			\sigma_1^2|v-v_*|^2 B\,  f(t,x,v)f(t,x,v_*)\dd\sigma\dd v_*\dd v\dd x\dd t
			\\
			&\qquad\qquad\qquad\qquad
			- \int_0^T\!\!\int_{\Omega_x^1\times\RR_v^3\times\RR_{v_*}^3\times\SSS_\sigma^2}\!\!\!
			(v_1 - u_1)^2 B\,  f(t,x,v)f(t,x,v_*)\dd\sigma\dd v_*\dd v\dd x\dd t
			\\
			&\qquad\qquad\qquad\qquad
			- \int_0^T\!\!\int_{\Omega_x^1\times\RR_v^3\times\RR_{v_*}^3\times\SSS_\sigma^2}\!\!\!
			(v_{*,1} - u_1)^2 B\,  f(t,x,v)f(t,x,v_*)\dd\sigma\dd v_*\dd v\dd x\dd t,
		\end{align*}
		where in the second equality we have used a pre-postcollisional change of variables, and in the last inequality we have applied Young's inequality. We can then apply the bounds on the terms \eqref{gain_term_a_priori_bound} and \eqref{loss_term_a_priori_bound} to get
		\begin{align*}
			\int_0^T\!\!
			\int\limits_{\Omega_x^1\times\RR_v^3\times\RR_{v_*}^3\times\SSS_\sigma^2}\!\!\!\!\!\!
			\sigma_1^2|v-v_*|^2 B\,  f(t,x,v)f(t,x,v_*)\dd\sigma\dd v_*\dd v\dd x\dd t
			\le
			C' + C' \int_0^t\! D_B(t)\dd t,
		\end{align*}
		for a new constant $C' = C'(\norm{f_{\mrm{in}}}_{L^1(1+|v|^2)},B)$. We would now like to bound the term
		\begin{align*}
			\int_{\SSS_\sigma^2}\! \sigma^2_1 
			B\left(|v-v_*|,\sigma\cdot\frac{v-v_*}{|v-v_*|}\right)\dd\sigma
			=: 
			\int_{\SSS_\sigma^2}\! \sigma^2_1 
			\beta(\sigma)\dd\sigma
		\end{align*}
		from below, using hypothesis \textbf{(H2)}, where we rewrite the collision kernel $B$ as $\beta(\sigma)$, temporarily suppressing the dependence on $v$ and $v_*$. But we can compute that
		\begin{align*}
			\int_{\SSS_\sigma^2}\! \sigma^2_1 
			\beta(\sigma)\dd\sigma
			&
			\ge
			\int_{\SSS_\sigma^2\cap\set{|\sigma_1|\ge\epsilon}}\! \sigma^2_1 
			\beta(\sigma)\dd\sigma
			\\
			&\ge
			\int_{\SSS_\sigma^2\cap\set{|\sigma_1|\ge\epsilon}}\! \epsilon^2
			\beta(\sigma)\dd\sigma
			\\
			&=
			\epsilon^2\left(
				\int_{\SSS_\sigma^2}\!\beta(\sigma)\dd\sigma
				-
				\int_{\SSS_\sigma^2\cap\set{|\sigma_1|<\epsilon}}\!\beta(\sigma)\dd\sigma
			\right)
			\\
			&\ge
			\epsilon^2\left(
				\int_{\SSS_\sigma^2}\!\beta(\sigma)\dd\sigma
				-
				4\pi\epsilon\sup_{\sigma\in\SSS^2}\beta(\sigma)
			\right)
			\\
			&\ge
			\epsilon^2
			\left(
				\int_{\SSS_\sigma^2}\!\beta(\sigma)\dd\sigma
				-
				\frac{4\pi\epsilon}{\delta}\int_{\SSS_\sigma^2}\!\beta(\sigma)\dd\sigma
			\right)
		\end{align*}
		holds for any $\epsilon>0$, where in the last inequality we have used hypothesis \textbf{(H2)}, and therefore by choosing $\epsilon = \delta/8\pi$, we get the bound
		\begin{align*}
			\int_{\SSS_\sigma^2}\! \sigma^2_1 
			B\left(|v-v_*|,\sigma\cdot\frac{v-v_*}{|v-v_*|}\right)\dd\sigma
			\ge
			\frac{C}{\delta^2}
			\int_{\SSS_\sigma^2}\!
			B\left(|v-v_*|,\sigma\cdot\frac{v-v_*}{|v-v_*|}\right)\dd\sigma
		\end{align*}
		for an absolute constant $C$, independent of $v$ and $v_*$. Therefore, we can bound
		\begin{align*}
			\int_0^T\!\!
			\int\limits_{\Omega_x^1\times\RR_v^3\times\RR_{v_*}^3\times\SSS_\sigma^2}\!\!\!\!\!\! 
			|v-v_*|^2 B\,  &f(t,x,v)f(t,x,v_*) \dd\sigma\dd v_*  \dd v\dd x\dd t
			\\
			&
			\le
			C' + C' \int_0^T D_B(t)\dd t
			\le
			\begin{cases}
				K_0 (1+T) & \Omega_x^1 = \TT_x^1
				\\
				K_0 & \Omega_x^1 = \RR_x^1
			\end{cases}
		\end{align*}
		for some constant $K_0 = K_0(\norm{f}_{L^1(1+|v|^2)},B)$, which concludes the proof.
	\end{proof}
	
	We can now follow arguments similar to those in the proof of Theorem \ref{Thm1}. By combining the bound \eqref{second_bound} with the hypothesis \textbf{(H2)}, we get
	\begin{align*}
			&\norm{S_q Q^+(f,f)}_X
			\\
			&\quad\le
			4\pi\delta^{-1}
			\int_{\TT_x^1}\!\int_{\RR_v^3\times\RR_{v_*}^3\times\mathbb{S}_\sigma^2}\!
			f(t,x,v)f(t,x,v_*) B\left(|v-v_*|,\sigma\cdot\frac{v-v_*}{|v-v_*|}\right)
			\\
			&\qquad\qquad\qquad\qquad\qquad\qquad\qquad\qquad\qquad\qquad\quad\cdot
			\left(
			1 + \frac{1}{q|v-v_*|}
			\right)
			\dd \sigma\dd v_*\dd v\dd x 
			\\
			&\quad\le
			4\pi\delta^{-1}(1+R_0^{-1})\left(1+\frac1q\right) A(t).
		\end{align*}
		Combining this with the estimate from Lemma \ref{ContEstLemv2} gives the bound
		\begin{align*}
			\norm{S_q Q^+(f,f)}_X
			\le
			\min\left(
			C_1\norm{f}_X^2,\ C_2 \left(1+\frac1q\right) A(t)
			\right),
		\end{align*}
		where $C_1$ and $C_2$ depend on the collision kernel $B$ only. As in \eqref{norm(f)_X_small_entropy_bound}, we can plug this into the Duhamel formulation \eqref{Duhamel_Boltzmann} and use non-negativity $f\ge 0$ to get
		\begin{align}
					\label{second_ODE_eq1}
				\norm{f(t)}_X 
				&\le
				\norm{S_{t-t_0} f(t_0)}_X + 
				\int_{t_0}^t\! \norm{S_{t-s} Q^+(f,f)(s)}_X\dd s
				\\
					\label{second_ODE_eq2}
				&\le
				\norm{f(t_0)}_X
				+
				\int_{t_0}^t\!
				\min\left(
				C_1\norm{f(s)}_X^2,\ C_2 \left(1+\frac{1}{t-s}\right) A(s)
				\right)
				\dd s
		\end{align}
		for any $0\le t_0\le t< T_*$. This integral inequality was introduced in \cite{BiryukCraigPanferov}, in which it was shown to imply $f\in L^\infty_{\mrm{loc}}([0,\infty),X)$, without a rate. We now prove sharp growth estimates for this inequality, which is the main technical contribution of this work.
	
		\begin{lemma}
		\label{Bony_ODE}
		Let $\varphi\in L^\infty_{\mrm{loc}}([0,T_*))$ be non-negative and locally bounded, and satisfy the integral inequality
		\begin{align}
			\label{nonlinear_Gronwall_prop_eq1}
			\varphi(t)\le \varphi(t_0)+ \int_{t_0}^t \min\left\{ c\varphi(s)^2, \left( 1+\frac{1}{t-s} \right) a(s) \right\} \dd s
		\end{align}
		for all $0\le t_0\le t < T_*$, for some $a\in L^1([0,T_*))$ with $a\ge0$. Then $\varphi$ satisfies the inequality
		\begin{align}
			\label{nonlinear_Gronwall_prop_eq2}
			\varphi(t)\le 2^{1+16 c\int_0^t a(s)\dd s} \left(\varphi(0) + \frac{1}{8 c} \right).
		\end{align}
	\end{lemma}
	
	\begin{proof}
		Our proof strategy relies on iterating  \eqref{nonlinear_Gronwall_prop_eq1} on intervals $[t_0,t)$ such that $\int_{t_0}^t a(s)\dd s $ is sufficiently small.
		
		We claim that if $\lambda = \int_{t_0}^t a(s)\dd s \le 1/16c$, then
		\begin{align}
			\label{iterative_Gronwall_inequality}
			\varphi(t)\le\frac{\varphi(t_0)+\lambda}{1-\sqrt{4c\lambda}}.
		\end{align}
		We first define  $\varphi_*(s) = \sup_{t_0\le r\le s} \varphi(r)$ such that $\varphi_*$ is increasing. We can now split the integral in  \eqref{nonlinear_Gronwall_prop_eq1} to get
		\begin{align*}
			\varphi(s) 
			&\le \varphi(t_0) + \int_{t_0}^{s-\delta/\varphi_*(s)} \left( 1+ \frac{1}{s-r}\right) a(r)\dd r
			+
			\int_{s-\delta/\varphi_*(s)}^s c\varphi(r)^2\dd r \\
			&\le
			\varphi(t_0) + \left(1+\frac{\varphi_*(s)}{\delta}\right)\lambda + c\delta \varphi_*(s),
		\end{align*}
		where we let $\de$ vary. We note that this bound remains true even when $\de/\overline\varphi(s)>s$, since this would imply
		\begin{align*}
			\int_{t_0}^s c\varphi(r)^2\dd r\le c\de\overline\varphi(s).
		\end{align*}
		We note that this bound is minimized by $\delta = \sqrt{\lambda/c}$, thus giving the bound
		\begin{align*}
			\varphi(s)\le \varphi(t_0) + \lambda + 2\sqrt{c\lambda}\varphi_*(s).
		\end{align*}
		Using that $\varphi_*$ is increasing, and choosing $r_*\in [t_0,s]$ such that
		\begin{align*}
			\varphi_*(s) = \sup_{r\in [t_0, s]} \varphi(r) = \varphi(r_*),
		\end{align*}
		we get the inequalities
		\begin{align*}
			\varphi_*(s) = \varphi(r_*)\le\varphi(t_0) + \lambda + 2\sqrt{c\lambda}\varphi_*(r_*)\le
			\varphi(t_0) + \lambda + 2\sqrt{c\lambda}\varphi_*(s).
		\end{align*}
		Rearranging this inequality in order to control $\varphi_*(s)$, we get
		\begin{align*}
			\varphi_*(s)\le \frac{\varphi(t_0) + \lambda}{1-\sqrt{4c\lambda}}.
		\end{align*}
		
		We note that whenever 
		\begin{align*}
			\lambda=\int_{t_0}^t a(s)\dd s\le \frac{1}{16c},
		\end{align*}
		this implies that 
		\begin{align*}
			\varphi(s)\le\varphi_*(s) \le 2 \left(\varphi(t_0) + \frac{1}{16 c} \right).
		\end{align*}
		
		In general, we will have $\int_0^{t} a(s)\dd s  > 1/16c$. However, standard properties of Lebesgue integrals give us that we can partition $[0,t)$ into $0 = t_0\le t_1\le \cdots \le t_{k-1} \le t_k = t$ such that
		\begin{align*}
			\int_{t_{i-1}}^{t_i} a(s) = \frac{1}{16c}
		\end{align*}
		for $1\le i \le k-1$, and $$\int_{t_{k-1}}^{t_k} a(s)\dd s\le \frac{1}{16c}.$$
		This then implies that 
		\begin{align*}
			\varphi(t_i) \le 2 \left(\varphi(t_{i-1}) + \frac{1}{16 c} \right)
		\end{align*}
		for all $1\le i \le k$. But iterating these bounds gives us
		\begin{align*}
			\varphi(t)=\varphi(t_k)
			&\le 2^k\varphi(0)+\frac{1}{16c}\sum_{i=1}^k 2^i \\
			&\le 2^k \left( \varphi(0) + \frac{1}{8c} \right).
		\end{align*}
		
		But since $\frac{k-1}{16c}\le\int_0^t a(s)\dd s$ by construction of the partition of $[0,t)$, we get the bound 
		\begin{align*}
			k\le 1+16c\int_0^t a(s)\dd s,
		\end{align*}	
		and therefore that
		\begin{align*}
			\varphi(t)\le 2^{1+16c\int_0^t a(s)\dd s} \left(\varphi(0) + \frac{1}{8c} \right),
		\end{align*}
		concluding the proof.
	\end{proof}
	
	By applying Lemma \ref{Bony_ODE} to \eqref{second_ODE_eq2}, we get the bound
	\begin{align}
		\nonumber
		\norm{f(t)}_X
		&\le
		e^{C_0(1 + \int_0^t A(s)\dd s)}\left( \norm{f_{\mrm{in}}}_X + C_1\right) 
		\\
		\begin{split}
		&\le
			\label{X_norm_growth_bound}
		\begin{cases}
			C e^{Ct} & \Omega_x^1 = \TT_x^1 \\
			C  & \Omega_x^1 = \RR_x^1
		\end{cases},
		\end{split}
	\end{align}
	for $C = C(\norm{f_{\mrm{in}}}_X,\norm{f_{\mrm{in}}}_{L^1(1+|v|^2)},B)$,
	where we have applied Lemma \ref{integral_of_A_lemma}. But this means that $f\in L^\infty([0,T),X)$ for any finite $T<T_*$, and therefore we have solutions for all time $f\in L^\infty_{\mrm{loc}}([0,\infty),X)$. It now remains to prove dissipation $\norm{f(t)}_X\to 0$ for finite energy data on the line, which we prove using the following lemma.

	\begin{lemma} \label{dissipation_lemma}
		Let $f$ be a solution to \eqref{1D_Boltzmann} in  $L^\infty([0,\infty),X)\cap C([0,\infty),W^{1,1}))\cap C_b([0,\infty),L^1(1+|v|^2))$ on $\RR_x^1\times\RR_v^3$. Then we have dissipation $\norm{f(t)}_X\to0$ as $t\to\infty$.
	\end{lemma}
	\begin{proof}
		
		We first must show that for any $g\in W^{1,1}(\RR_x^1\times\RR_v^3)$, we have the dispersive estimate
		\begin{align}\label{X_norm_transport_dispersion}
			\norm{S_t g}_X\le\frac1t \norm{g}_{W^{1,1}_{x,v}}.
		\end{align}
		But we can rewrite $\norm{S_t g}_X = \sup_{q\ge t}\norm{S_q g}_{L^\infty_x L^1_v}$, and for all $q\ge t$ we have
		\begin{align*}
			\norm{S_q g}_{L^\infty_x L^1_v}
			&=
			\operatorname*{ess.sup}\limits_{x\in\RR^1}
				\int_{\RR_v^3}\! g(x-q v_1,v)\dd v
			\\
			&\le
			\operatorname*{ess.sup}\limits_{x\in\RR^1}
				\frac1q \int_\RR\!\int_{\RR^2} g\left( z,\frac{x-z}{q},v_2,v_3\right)\dd v_2 \dd v_3\dd z
			\\
			&\le
			\frac1q 
				\int_\RR\!\int_{\RR^2} \norm{g(z,\cdot,v_2,v_3)}_{L^\infty(\RR_{v_1}^2)}\dd v_2 \dd v_3\dd z
			\\
			&\le \frac1q \norm{g}_{W^{1,1}_{x,v}}
			\le \frac1t\norm{g}_{W^{1,1}_{x,v}},
		\end{align*}
		which proves the estimate \eqref{X_norm_transport_dispersion}, where we have used the change of variables $z = x - q v_1$. We can now apply this estimate to \eqref{second_ODE_eq1} to get
		\begin{align*}
			\norm{f(t)}_X
			&\le\norm{S_{t-t_0}f(t_0)}_X 
			+ \int_{t_0}^t\!\norm{S_{t-s}Q^+(f,f)(s)}_X\dd s
			\\
			&\le
			(t-t_0)^{-1}\norm{f(t_0)}_{W^{1,1}}
			+ 
			\int_{t_0}^t\!
			\min\left\{
			C\norm{f(s)}_X^2,\frac{C}{t-s}A(s)
			\right\}
			\dd s
			\\
			&\le
			(t-t_0)^{-1}C(t_0)
			+
			\int_{t_0}^t\! \min\left\{
			C\norm{f}_{L^\infty_t X}^2,\frac{C}{t-s}A(s)
			\right\}\dd s,
		\end{align*}
		where we set $C(t_0) = \norm{f(t_0)}_{W^{1,1}_{x,v}}$. If we suitably choose $\de>0$, we can bound the last integral as
		\begin{align*}
			\int_{t_0}^t\! \min\left\{
			C\norm{f}_{L^\infty_t X}^2,\frac{C}{t-s}A(s)
			\right\}\dd s
			&\le
			\int_{t_0}^{t-\de}\! \frac{C}{t-s}A(s) \dd s
			+
			\int_{t-\de}^t\! C\norm{f}_{L^\infty_t X}^2 \dd s
			\\
			&\le
			\frac{C}{\de} \int_{t_0}^t\! A(s)\dd s
			+
			C\de\norm{f}_{L^\infty_t X}^2
			\\
			&= 2 C\norm{f}_{L^\infty_t X} \left( \int_{t_0}^t\! A(s)\dd s\right)^{\frac12},
		\end{align*}
		where in the last line we have taken
		\begin{align*}
			\de = \norm{f}_{L^\infty_t X}^{-1}\left(\int_{t_0}^t\! A(s)\dd s\right)^{\frac12}.
		\end{align*}
		We can therefore bound
		\begin{align*}
			\norm{f(t)}_X
			&\le 
			(t-t_0)^{-1}C(t_0)
			+
			2C\norm{f}_{L^\infty_t X}\left( \int_{t_0}^t\! A(s)\dd s\right)^{1/2}
			\\
			&\le
			(t-t_0)^{-1}C(t_0)
			+
			C_1\ 
			\left( \int_{t_0}^\infty\! A(s)\dd s\right)^{\frac12}.
		\end{align*} 
		where $C_1 = C_1(B,\norm{f_{\mrm{in}}}_{L^1(1+|v|^2)},\norm{f_{\mrm{in}}}_X)$ is independent of time. But since 
		\begin{align*}
			\int_{t_0}^\infty\! A(s)\dd s\to 0\qquad\hbox{as}\quad t_0\to\infty
		\end{align*}
		by Lemma \ref{integral_of_A_lemma}, we see that by taking $t_0$ sufficiently large, we get convergence $\norm{f(t)}_X\to 0$ as $t\to\infty$.
	\end{proof}
	
	Since $f_{\mrm{in}}\in W^{1,1}\cap L^1(1+|v|^2)$ implies $f\in C([0,T_*),W^{1,1})\cap C_b([0,T_*),L^1(1+|v|^2))$ by Proposition \ref{regularity_proposition}, and since we have shown that $T_*=\infty$ and $f\in L^\infty([0,\infty),X)$ in \eqref{X_norm_growth_bound}, the conditions of Lemma \ref{dissipation_lemma} apply, which concludes the proof of Theorem \ref{Thm2}.

	\bigskip
	\noindent
	{\bf{Acknowledgments}}: I would like to thank my advisor, Cl\'ement Mouhot, for proposing this problem, for his guidance, and for the many helpful suggestions he gave over the course of working on this problem. I would also like to thank Jo Evans, for our conversations about this problem, and for pointing out a way to strengthen the result. This article was supported by the ERC grant MAFRAN 2017-2022, and by the Cantab Capital Institute for the Mathematics of Information.

	\bibliographystyle{acm}
	\bibliography{Boltzmann}

\end{document}